\documentclass[psamsfonts,reqno]{amsart}
\usepackage{amssymb,eucal,latexsym}

\newcommand{\pup}[1]{\textup{(}{#1}\textup{)}}

\newcommand{\jz}{$(\vee,0)$}
\newcommand{\mup}{$(\wedge,1)$}

\newcommand{\jzs}{\jz-semi\-lat\-tice}
\newcommand{\mus}{\mup-semi\-lat\-tice}

\newcommand{\jzh}{\jz-ho\-mo\-mor\-phism}
\newcommand{\mh}{meet-ho\-mo\-mor\-phism}

\newcommand{\js}{join-sem\-i\-lat\-tice}
\newcommand{\jh}{join-ho\-mo\-mor\-phism}

\newcommand{\jirr}{join-ir\-re\-duc\-i\-ble}
\newcommand{\jirry}{join-ir\-re\-duc\-i\-bil\-i\-ty}
\newcommand{\mirr}{meet-ir\-re\-duc\-i\-ble}
\newcommand{\jsd}{join-sem\-i\-dis\-trib\-u\-tive}
\newcommand{\jsdy}{join-sem\-i\-dis\-trib\-u\-tiv\-i\-ty}

\newcommand{\bdl}{bounded distributive lattice}
\newcommand{\bl}{Boolean lattice}
\newcommand{\res}{\mathbin{\restriction}}

\newcommand{\contr}{a contradiction}
\DeclareMathOperator{\rK}{K}
\DeclareMathOperator{\rL}{L}

\newcommand{\two}{\mathbf{2}}

\newcommand{\ol}[1]{\overline{#1}}
\newcommand{\set}[1]{\left\{#1\right\}}
\newcommand{\setm}[2]{\set{{#1}\mid{#2}}}
\newcommand{\famm}[2]{\left({#1}\mid{#2}\right)}

\newcommand{\SDjn}[1]{\ensuremath{(\mathrm{SD}_{\vee}^{#1})}}
\newcommand{\Co}{\operatorname{\mathbf{Co}}}

\newcommand{\eps}{\varepsilon}
\newcommand{\dnw}{\mathbin{\downarrow}}
\newcommand{\ddnw}{\mathbin{\downdownarrows}}
\newcommand{\upw}{\mathbin{\uparrow}}
\newcommand{\JJd}{\bigvee\nolimits^{\uparrow}}

\DeclareMathOperator{\End}{End}

\newcommand{\bD}{\mathbin{\boldsymbol{D}}}
\newcommand{\bL}{\mathbin{\boldsymbol{L}}}
\newcommand{\bP}{\mathbin{\boldsymbol{P}}}

\newcommand{\cD}{\boldsymbol{\mathcal{D}}}
\newcommand{\cK}{\boldsymbol{\mathcal{K}}}
\newcommand{\cV}{\boldsymbol{\mathcal{V}}}
\DeclareMathOperator{\cov}{Cov}
\DeclareMathOperator{\icov}{iCov}
\DeclareMathOperator{\tcov}{tCov}
\DeclareMathOperator{\mcov}{mCov}

\newcommand{\cX}{\mathcal{X}}

\newcommand{\proj}[2]{{#1}_{(#2)}}

\newcommand{\fp}{\mathfrak{p}}
\newcommand{\fq}{\mathfrak{q}}
\newcommand{\fr}{\mathfrak{r}}
\newcommand{\PP}{\mathbb{P}}
\newcommand{\ZZ}{\mathbb{Z}}

\DeclareMathOperator{\Id}{Id}
\DeclareMathOperator{\Fil}{Fil}

\DeclareMathOperator{\col}{col}

\DeclareMathOperator{\con}{con}
\DeclareMathOperator{\J}{J}
\DeclareMathOperator{\Jc}{J_c}
\DeclareMathOperator{\At}{At}
\newcommand{\veec}{\mathbin{\vee_{\mathrm{c}}}}
\newcommand{\leref}{\leq_{\mathrm{ref}}}
\newcommand{\lsref}{<_{\mathrm{ref}}}
\newcommand{\es}{\varnothing}

\newcommand{\xp}{\mathbf{p}}

\newcommand{\xs}{\mathbf{s}}
\newcommand{\xt}{\mathbf{t}}

\newcommand{\vx}{\mathsf{x}}
\newcommand{\vy}{\mathsf{y}}
\newcommand{\vz}{\mathsf{z}}
\newcommand{\vt}{\mathsf{t}}

\numberwithin{equation}{section}

\theoremstyle{plain}

\newtheorem{theorem}{Theorem}[section]
\newtheorem{proposition}[theorem]{Proposition}
\newtheorem{corollary}[theorem]{Corollary}
\newtheorem{lemma}[theorem]{Lemma}

\newtheorem*{sclaim}{Claim}

\theoremstyle{definition}

\newtheorem{definition}[theorem]{Definition}
\newtheorem{notation}[theorem]{Notation}
\newtheorem{problem}{Problem}

\theoremstyle{remark}

\newcommand{\qedc}{{\qed}~{\rm Claim~{\theclaim}.}}
\newcommand{\qedsc}{{\qed}~{\rm Claim.}}

\newenvironment{scproof}
{\begin{proof}[Proof of Claim.]}
{\qedsc\renewcommand{\qed}{}\end{proof}}

\title{Varieties of lattices with geometric descriptions}

\author[L. Santocanale]{Luigi Santocanale}
\address{Laboratoire d'Informatique Fondamentale de Marseille\\
Universit\'e de Provence\\
39 rue F. Joliot Curie\\
13453 Marseille Cedex 13\\
France}
\email{luigi.santocanale@lif.univ-mrs.fr}
\urladdr{http://www.lif.univ-mrs.fr/\~{}lsantoca/}

\author[F. Wehrung]{Friedrich Wehrung}
\address{LMNO, CNRS UMR 6139\\
D\'epartement de Math\'ematiques\\
Universit\'e de Caen\\
14032 Caen Cedex\\
France}
\email{wehrung@math.unicaen.fr}
\urladdr{http://www.math.unicaen.fr/\~{}wehrung}

\date{\today}

\thanks{Both authors were partially supported by the PEPS project TRECOLOCOCO}

\begin{document}

\begin{abstract}
  A lattice~$L$ is \emph{spatial} if every element of~$L$ is a join of
  completely \jirr\ elements of~$L$ (\emph{points}), and
  \emph{strongly spatial} if it is spatial and the minimal coverings
  of completely \jirr\ elements are well-behaved. Herrmann, Pickering,
  and Roddy proved in 1994 that every \emph{modular} lattice can be
  embedded, within its variety, into an algebraic and spatial lattice.
  We extend this result to $n$-distributive lattices, for fixed~$n$.
  We deduce that the variety of all $n$-distributive lattices is
  generated by its finite members, thus it has a decidable word
  problem for free lattices.
  This solves two problems stated by Huhn in 1985.
  We prove that every modular
  (resp., $n$-distributive) lattice embeds within its variety into
  some \emph{strongly spatial} lattice. Every lattice which is either
  algebraic modular spatial or bi-algebraic is strongly spatial.
  
  We also construct a lattice that cannot be embedded, within its
  variety, into any algebraic and spatial lattice. This lattice has a
  least and a largest element, and it generates a locally finite
  variety of \jsd\ lattices.
\end{abstract}

\maketitle

\section{Introduction}\label{S:Intro}

An element~$p$ in a lattice~$L$ is \emph{completely \jirr}, or a
\emph{point}, if there is a largest element smaller than~$p$. We say
that~$L$ is \emph{spatial} if every element of~$L$ is a (possibly
infinite) join of points. In such a case, elements of~$L$ can be
identified with certain sets of points of~$L$. If, in addition, $L$ is
\emph{algebraic}, then we say that we have a \emph{geometric
  description} of~$L$. When dealing with equational properties of
lattices, the geometric description enables to prove representation
results that may have been very hard to obtain otherwise.

A prominent illustration of such methods is given in Herrmann,
Pickering, and Roddy~\cite{HPR}, where it is proved that every
\emph{modular} lattice~$L$ embeds into some algebraic and spatial
lattice~$\ol{L}$ that satisfies the same identities as~$L$---we say
that~$L$ embeds into~$\ol{L}$ \emph{within its variety}. In
particular, as~$L$ is modular, so is~$\ol{L}$. This is used
in~\cite{HPR} to prove  that \emph{a lattice~$L$ embeds into the subspace lattice of a vector space over an arbitrary field if{f}~$L$ is modular and $2$-distributive}.

Nevertheless it was not known whether \emph{every} lattice embeds,
within its variety, into some algebraic and spatial lattice. (This
question is stated in the comments following Semenova and
Wehrung~\cite[Problem~4]{CoP3}.) We find a counterexample to that
question in Theorem~\ref{T:NoVarEmb}. This counterexample is
\emph{\jsd}---in fact, it generates a variety of \jsd\ lattices.

Yet even for \jsd\ lattices, there are many situations where lattices
enjoy geometric descriptions. Such geometric descriptions are
massively used in Semenova and Wehrung~\cite{CoP1,CoP2,CoP3} or Semenova and Zamojska-Dzienio~\cite{SeZa1} for descriptions of
lattices of order-convex subsets of various kinds of posets. Denote by
$\Co(P)$ the lattice of all order-convex subsets of a poset~$P$, and
by~$\mathbf{SUB}$ the class of all lattices that can be embedded into
some~$\Co(P)$. It is proved in Semenova and Wehrung~\cite{CoP1}
that~$\mathbf{SUB}$ is a finitely based variety of lattices. It is
asked in Semenova and Wehrung~\cite[Problem~4]{CoP3} whether every
lattice in~$\mathbf{SUB}$ can be embedded, within its variety, into
some algebraic and spatial lattice. We prove in the present paper that
this is the case---we actually get a stronger result
(Theorem~\ref{T:NatRefndistr}): \emph{For every positive integer~$n$,
  every $n$-distributive lattice can be embedded, within its variety,
  into some algebraic and spatial lattice}. As all lattices of the
form~$\Co(P)$ are $2$-distributive, this solves the question above.
The lattices obtained in Theorem~\ref{T:NatRefndistr} are actually
\emph{strongly spatial} (cf. Definition~\ref{D:Seed}), which means
spatial plus the fact that minimal join-covers between points are well-behaved.
For \emph{modular} lattices, spatial implies
strongly spatial (cf.  Theorem~\ref{T:ModSp2StrSp}) and so the
distinction is immaterial.

A further consequence of Theorem~\ref{T:NatRefndistr} is that
\emph{the variety of all $n$-distributive lattices is generated by its
  finite members, thus it has a decidable word problem for free
  lattices} (Theorem~\ref{T:Finn-distr}). This solves two problems contained in Huhn \cite{Huhn85,Huhn88}.

We also prove that every lattice which is either well-founded or bi-algebraic is strongly spatial (Corollary~\ref{C:BiAlg2StrSp}). Hence our main counterexample (Theorem~\ref{T:NoVarEmb}) extends the result, proved by the second author in~\cite{DLLB}, that not every lattice can be embedded into some bi-algebraic lattice.

\section{Basic concepts}\label{S:Basic}

We set
 \begin{align*}
 Q\dnw X&:=\setm{q\in Q}{(\exists x\in X)(q\leq x)}\,,\\
 Q\upw X&:=\setm{q\in Q}{(\exists x\in X)(q\geq x)}\,,\\
 Q\ddnw X&:=\setm{q\in Q}{(\exists x\in X)(q<x)}\,,
 \end{align*}
for all subsets~$X$ and~$Q$ in a poset~$P$. We also set $Q\dnw a:=Q\dnw\set{a}$, $Q\upw a:=Q\upw\set{a}$, and $Q\ddnw a:=Q\ddnw\set{a}$, for each $a\in P$. A subset~$X$ of~$P$ is a \emph{lower subset} of~$P$ if $X=P\dnw X$.

A subset~$X$ of~$P$ \emph{refines} a subset~$Y$ of~$P$, in notation $X\leref Y$, if $X\subseteq P\dnw Y$. (As the present work touches upon algebraic and continuous lattices, where $x\ll y$ denotes the ``way-below'' relation, we shall stray away from the usual notation $X\ll Y$ for the refinement relation on subsets.) We shall also write $X\lsref Y$ for the conjunction of~$X\leref Y$ and~$X\neq Y$.

For elements~$x$ and~$y$ in a poset~$P$, let $x\prec y$ hold (in words, ``$x$ is a lower cover of~$y$'', or ``$y$ is an upper cover of~$x$'') if $x<y$ and there is no~$z\in P$ such that $x<z<y$. An element~$p$ in a \js~$L$ is
\begin{itemize}
\item[---] \emph{\jirr} if $p=\bigvee X$ implies that $p\in X$, for any finite subset~$X$ of~$L$ (in particular, taking $X:=\es$, this rules out~$p$ being the zero element of~$L$);
  
\item[---] \emph{completely \jirr}---from now on a \emph{point}, if
  the set of all elements smaller than~$p$ has a largest element, then
  denoted by~$p_*$. Of course, every point is \jirr;

\item[---] an \emph{atom} of~$L$ if~$L$ has a zero element and $0\prec p$.
\end{itemize}
We denote by~$\J(L)$ ($\Jc(L)$, $\At(L)$, respectively) the set of all \jirr\ elements (points, atoms, respectively) of~$L$. Trivially, $\At(L)\subseteq\Jc(L)\subseteq\J(L)$.

We say that a subset $\Sigma$ of~$L$ is \emph{join-dense} in~$L$ if every element of~$L$ is a join of elements of~$\Sigma$. Equivalently, for all $a,b\in L$ with $a\nleq b$, there exists $x\in\Sigma$ such that $x\leq a$ and $x\nleq b$. An element~$a$ in~$L$ is \emph{compact} (resp., \emph{countably compact}) if for every nonempty directed (resp., countable nonempty directed) subset~$D$ of~$L$ with a join, $a\leq\bigvee D$ implies that $a\in L\dnw D$. We say that~$L$ is
\begin{itemize}
\item[---] \emph{spatial} if the set of all points of~$L$ is join-dense in~$L$;

\item[---] \emph{atomistic} if the set of all atoms of~$L$ is join-dense in~$L$;

\item[---] \emph{compactly generated} if the set of all compact elements of~$L$ is join-dense in~$L$;

\item[---] \emph{algebraic} if it is complete and compactly generated;

\item[---] \emph{bi-algebraic} if it is both algebraic and dually algebraic. 
\end{itemize}
It is well known that every dually algebraic lattice is spatial---see Gierz \emph{et al.}~\cite[Theorem~I.4.25]{Comp} or Gorbunov~\cite[Lemma~1.3.2]{Gorb}.

A lattice~$L$ is \emph{upper continuous} if the equality $a\wedge\bigvee D=\bigvee(a\wedge D)$ (where $a\wedge D:=\setm{a\wedge x}{x\in D}$) holds for every $a\in L$ and every nonempty directed subset~$D$ of~$L$ with a join. Lower continuity is defined dually.

\begin{proposition}[folklore]\label{P:BasicCompGen}
Every compactly generated lattice~$L$ is upper continuous and every point of~$L$ is compact.
\end{proposition}

\begin{proof}
Let $a\in L$ and let~$D$ be a nonempty directed subset of~$L$ with a join, we must prove that $a\wedge\bigvee D\leq\bigvee(a\wedge D)$. Let $c\in L$ compact with $c\leq a\wedge\bigvee D$. As~$c\leq\bigvee D$ and~$c$ is compact, there exists $d\in D$ such that $c\leq d$. It follows that $c\leq a\wedge d\leq\bigvee(a\wedge D)$. As~$L$ is compactly generated, the upper continuity of~$L$ follows.

Now let~$p$ be a point of~$L$ and let~$D$ be a directed subset of~$L$ with a join such that $p\leq\bigvee D$. If $p\notin L\dnw D$, then $p\wedge x\leq p_*$ for each $x\in D$, thus, by using the upper continuity of~$L$, we get $p=p\wedge\bigvee D\leq p_*$, \contr.
\end{proof}

Let~$L$ be a \jzs. We denote by~$\Sigma^\vee$ the set of all finite joins of elements of~$\Sigma$ ($0$ included), for each subset $\Sigma\subseteq L$. Furthermore, we denote by~$\Id L$ the ideal lattice of~$L$. Dually, for a \mus~$L$, we denote by~$\Fil L$ the lattice of all filters (i.e., dual ideals) of~$L$, partially ordered under reverse inclusion.

A lattice is
\begin{itemize}
\item \emph{$n$-distributive} (where $n$ is a positive integer) if it satisfies the identity
 \begin{equation}\label{Eq:ndistr}
 \vx\wedge\bigvee_{0\leq i\leq n}\vy_i=
 \bigvee_{0\leq i\leq n}
 \biggl(\vx_i\wedge\bigvee_{0\leq j\leq n,\ j\neq i}\vy_j\biggr)\,,
 \end{equation}

\item \emph{\jsd} if it satisfies the quasi-identity
 \[
 \vx\vee\vy=\vx\vee\vz\ \Longrightarrow\ \vx\vee\vy=\vx\vee(\vy\wedge\vz)\,.
 \]
\end{itemize}

We define a sequence $\famm{\xp_n}{n<\omega}$ of ternary lattice terms by
 \begin{align*}
 \xp_0(\vx,\vy,\vz)&:=\vy\,,\\
 \xp_{n+1}(\vx,\vy,\vz)&:=\vy\wedge(\vx\vee\xp_n(\vx,\vz,\vy))\,,
 &&\text{for all }n<\omega\,. 
 \end{align*}
Observe that the lattice inclusions
 $\vy\wedge\vz\leq\xp_n(\vx,\vy,\vz)\leq\vy$, for $n<\omega$, are
 valid in all lattices. Denote by \SDjn{n} the lattice identity
 $\xp_n(\vx,\vy,\vz)\leq\vx\vee(\vy\wedge\vz)$. In particular,
 distributivity is equivalent to \SDjn{1}, and \SDjn{n} implies
 \SDjn{n+1}. If a lattice satisfies \SDjn{n} for some~$n$, then it is
 \jsd; the converse holds for finite lattices, see J\'onsson and Rival
 \cite{JoRi79} or Jipsen and Rose~\cite[Section~4.2]{JiRo}.

We shall denote by $\con_L(x,y)$ the least congruence of a lattice~$L$ that identifies elements~$x,y\in L$. For a congruence~$\theta$ of~$L$, we shall write $x\equiv_\theta y$ instead of $(x,y)\in\theta$ and $x\leq_\theta y$ instead of $(x\vee y,y)\in\theta$.

We shall denote by~$\two=\set{0,1}$ the two-element lattice and by~$\omega$ the chain of all non-negative integers. For any poset~Ê$P$, we shall set $P^-:=P\setminus\set{0}$ if~$P$ has a least element, and $P^-:=P$ otherwise.

\section{Seeds, algebraic lattices, (strongly) spatial lattices}\label{S:seed}

\begin{definition}\label{D:MCRP}
For an element~$p$ and a finite subset~$E$ in a \js~$L$, we say that
\begin{itemize}
\item $E$ \emph{covers} (resp., \emph{joins to}) $p$ if $p\leq\bigvee E$ (resp., $p=\bigvee E$). We set
 \begin{align*}
 \cov(p)&:=\setm{E\subseteq L\text{ finite}}{E\text{ covers }p}\,,\\
 \cov_=(p)&:=\setm{E\subseteq L\text{ finite}}{E\text{ joins to }p}\,.
 \end{align*}
The elements of $\cov(p)$ are called the \emph{join-covers of~$p$}, while the elements of $\cov_=(p)$ are called the \emph{join-representations of~$p$}.

\item $E$ \emph{covers} (resp., \emph{joins to}) $p$ \emph{irredundantly} if~$E$ covers~$p$ (resp., joins to~$p$) and $p\nleq\bigvee(E\setminus\set{u})$ for each $u\in E$. Observe that both conditions imply that $E\subseteq L^-$. We set
 \begin{align*}
 \icov(p)&:=\setm{E\subseteq L^-\text{ finite}}{E\text{ covers }p
 \text{ irredundantly}}\,,\\
 \icov_=(p)&:=\setm{E\subseteq L^-\text{ finite}}{E\text{ joins to }p
 \text{ irredundantly}}\\
& \hspace{10mm}=\icov(p)\cap\cov_=(p)\,.
 \end{align*}
The elements of $\icov(p)$ are called the \emph{irredundant join-covers of~$p$}, while the elements of $\icov_=(p)$ are called the \emph{irredundant join-representations of~$p$}.

\item $E$ \emph{covers} (resp., \emph{joins to}) $p$ \emph{tightly} if $E\subseteq L^-$, $E$ covers~$p$ (resp., joins to~$p$), and $p\nleq x\vee\bigvee(E\setminus\set{u})$ for each $u\in E$ and each $x<u$. We set
 \begin{align*}
 \tcov(p)&:=\setm{E\subseteq L^-\text{ finite}}{E\text{ covers }p
 \text{ tightly}}\,,\\
 \tcov_=(p)&:=\setm{E\subseteq L^-\text{ finite}}{E\text{ joins to }p
 \text{ tightly}}\\
& \hspace{10mm}=\tcov(p)\cap\cov_=(p)\,. 
 \end{align*}
The elements of $\tcov(p)$ are called the \emph{tight join-covers of~$p$}, while the elements of $\tcov_=(p)$ are called the \emph{tight join-representations of~$p$}.

\item $E$ \emph{covers} (resp., \emph{joins to}) $p$ \emph{minimally} if $E\subseteq L$, $E$ covers~$p$ (resp., joins to~$p$), and $p\leq\bigvee X$ and $X\leref E$ implies that $E\subseteq X$ for each finite subset~$X$ of~$L$. Observe that both conditions imply that $E\subseteq L^-$. We set
 \begin{align*}
 \mcov(p)&:=\setm{E\subseteq L^-\text{ finite}}{E\text{ covers }p
 \text{ minimally}}\,,\\
 \mcov_=(p)&:=\setm{E\subseteq L^-\text{ finite}}{E\text{ joins to }p
 \text{ minimally}}\\
& \hspace{10mm}=\mcov(p)\cap\cov_=(p)\,.
 \end{align*}
The elements of $\mcov(p)$ are called the \emph{minimal join-covers of~$p$}, while the elements of $\mcov_=(p)$ are called the \emph{minimal join-representations of~$p$}.
\end{itemize}

Minimality comes in more than one way in defining~$\mcov(p)$:

\begin{lemma}[folklore]\label{L:MinCovDef}
The minimal join-covers of~$p$ are exactly the $\leref$-minimal elements of~$\icov(p)$.
\end{lemma}

We define similarly irredundant covers for \emph{vectors}, by saying, for example, that a family $\famm{a_i}{i\in I}$ (where~$I$ is finite) covers~$p$ irredundantly if the \emph{set} $\setm{a_i}{i\in I}$ covers~$p$ irredundantly and the map $i\mapsto a_i$ is one-to-one. A similar definition applies to tight, resp. minimal, covers, resp. join-representations.
\end{definition}

Observe that we allow the possibility of \emph{trivial join-covers} $p\leq\bigvee X$, that is, those~$X$ such that $p\in L\dnw X$.

It is a straightforward exercise to verify that the containments
 \[
 \mcov(p)\subseteq\tcov(p)\subseteq\icov(p)\subseteq\cov(p)
 \]
hold, with none of the converse containments holding as a rule, with easy counterexamples for finite lattices. A similar comment applies to the containments $\mcov_=(p)\subseteq\tcov_=(p)\subseteq\icov_=(p)\subseteq\cov_=(p)$.

\begin{lemma}[folklore]\label{L:MJCjirr}
Let~$p$ be an element in a \js~$L$. Then every element~$E$ of~$\mcov(p)$ is contained in~$\J(L)$. Furthermore, if~$p$ is compact, then~$E$ is contained in~$\Jc(L)$.
\end{lemma}

\begin{proof}
If one of the elements of~$E$ is the zero of~$L$, the set
$X:=E\setminus\set{0}$ belongs to~$\cov(a)$ and refines~$E$, thus
contains~$E$, \contr; so all elements of~$E$ are nonzero. If $q\in E$
is not \jirr, then $q=x\vee y$ for some $x,y<q$, so the set
$X:=\set{x,y}\cup(E\setminus\set{q})$ refines~$E$ and belongs to~$\cov(a)$, and so $E\subseteq X$, and thus $q\in X$, \contr. Therefore, $E\subseteq\J(L)$.

Now assume that~$p$ is compact and that an element~$q$ of~$E$ is not a point. As~$q$ is \jirr, the set $D:=L\ddnw q$ is nonempty directed with join~$q$. Setting $b:=\bigvee(E\setminus\set{q})$, it follows that $p\leq\bigvee E=b\vee\bigvee D$, thus, as~$p$ is compact, there exists $x\in D$ such that $p\leq b\vee x$. It follows that the set $X:=\set{x}\cup(E\setminus\set{q})$ belongs to~$\cov(p)$, thus, as~$X$ refines~$E$, we get $E\subseteq X$, and thus $q\in X$, \contr.
\end{proof}

The following result originates in the proof of Nation~\cite[Theorem~3.2]{Nation90}.

\begin{lemma}\label{L:Irr2tight}
Let~$p$ be an element in a complete and lower continuous lattice~$L$. Then every element of~$\cov(p)$ can be refined to some element of~$\tcov(p)$.
\end{lemma}

\begin{proof}
Let $E=\set{p_1,\dots,p_m}\in\cov(p)$, with $m<\omega$, and set 
 \[
 \cX:=\Bigl\{(x_1,\dots,x_m)\in L^m\mid
 p\leq\bigvee_{i=1}^mx_i\text{ and }x_i\leq p_i\text{ for each }i\Bigr\}\,,
 \]
ordered componentwise. It follows from the completeness and lower continuity of~$L$ that every nonempty chain of~$\cX$ has a meet, which also belongs to~$\cX$. By Zorn's Lemma, $\cX$ has a minimal element, say $(q_1,\dots,q_m)$. Then $F:=\set{q_1,\dots,q_m}\setminus\set{0}$ refines~$E$, and it belongs to~$\tcov(p)$.
\end{proof}

In case~$p$ is a nonzero element in an infinite complete atomless Boolean lattice~$L$, then $\mcov(p)$ is empty. In particular, $\tcov(p)$ cannot be replaced by~$\mcov(p)$ in the statement of Lemma~\ref{L:Irr2tight}. However, searching in which particular cases this can be done leads to fascinating problems. Partial (and nontrivial) positive answers are introduced in Theorem~\ref{T:BiAlg2StrSp} and Lemma~\ref{L:NatRefndistr}.

\emph{Join-seeds} were introduced in Semenova and Wehrung~\cite{CoP3}, in a context where all lattices were $2$-distributive. The \emph{seeds} that we define here are related.

\begin{definition}\label{D:Seed}
A subset~$\Sigma$ in a \js\ $L$ is
 \begin{itemize}
 \item[---] a \emph{pre-seed} if for each $p\in\Sigma$ and each $X\in\cov(p)$, there exists $I\in\cov(p)$ contained in~$\Sigma$ such that $I\leref X$;

 \item[---] a \emph{quasi-seed} if it is a join-dense pre-seed contained in~$\J(L)$;
  
 \item[---] a \emph{seed} if it is join-dense, contained in~$\J(L)$, and for each $p\in\Sigma$ and each $X\in\cov(p)$, there exists $I\in\mcov(p)$ contained in~$\Sigma$ such that $I\leref X$. 
 \end{itemize}
We say that~$L$ is \emph{strongly spatial} if~$\Jc(L)$ is a seed in~$L$.
\end{definition}

As a seed is join-dense, every strongly spatial lattice is spatial. In the \emph{distributive} case, the two statements are equivalent: indeed, every algebraic, distributive, and spatial lattice is obviously strongly spatial (\jirr\ elements have no nontrivial join-covers). This is also the case for \emph{modular} lattices, but the proof is harder, see Theorem~\ref{T:ModSp2StrSp}. Another easy case of strong spatiality is provided by the following result.

\begin{proposition}\label{P:AlgAt2StrSp}
Every algebraic atomistic lattice is strongly spatial.
\end{proposition}

\begin{proof}
The points of an algebraic atomistic lattice~$L$ are exactly its atoms. Now if~$A$ is a finite cover of an atom~$p$ of~$L$, it follows from the compactness of~$p$ together with the fact that each element of~$A$ is a join of atoms that there exists a finite cover~$X$ of~$p$, consisting only of atoms, refining~$A$. Now every irredundant cover~$Y$ of~$p$ contained in~$X$ refines~$A$ and belongs to~$\mcov(p)$.
\end{proof}

It is easy to see that in the non-modular case, algebraic and spatial does not imply strongly spatial. For example, let $\omega^{\partial}:=\setm{n^*}{n<\omega}$ with $0^*>1^*>2^*>\cdots$. Then the lattice
 \[
 L:=\omega^{\partial}\cup\set{0,c}\,,
 \]
with the only new relations $0<c<0^*$ and $0<n^*$ for each $n<\omega$, is algebraic and spatial although not strongly spatial. This example is $2$-distributive, and not dually algebraic. The latter observation is also a consequence of the forthcoming Corollary~\ref{C:BiAlg2StrSp}. In order to prepare for the proof of that result, we shall establish a few lemmas with independent interest. {}From Lemma~\ref{L:SumRef} to Proposition~\ref{P:Interp} we shall fix an element~$p$ in a \js~$L$.

\begin{lemma}\label{L:SumRef}
  Let $A_0,A_1\in\tcov(p)$ such that $A_1\leref A_0$. Then for each
  $a\in A_0$, the set $A_1\dnw a$ joins to $a$ tightly. Moreover, each element
  of~$A_0$ contains an element of~$A_1$ and $\bigvee A_0=\bigvee A_1$.
\end{lemma}

\begin{proof}
Set $\ol{a}:=\bigvee(A_1\dnw a)$, for each $a\in A_0$.
The assumption that~$A_1$ refines~$A_0$ means that $A_1=\bigcup\famm{A_1\dnw a}{a\in A_0}$, thus $p\leq\bigvee A_1=\bigvee\famm{\ol{a}}{a\in A_0}$. As $\ol{a}\leq a$ for each~$a$ and as $A_0\in\tcov(p)$, it follows that $\ol{a}=a$ for each $a\in A_0$. In particular, $a$ contains an element of~$A_1$.

Now let $a\in A_0$ and $b\in A_1\dnw a$, and let $y\leq b$ such that
$a=y\vee\bigvee((A_1\dnw a)\setminus\set{b})$. Set
$B:=A_1\setminus(A_1\dnw a)$. As
 \[
 p\leq\bigvee A_1=\bigvee(A_1\dnw a)\vee\bigvee B=a\vee\bigvee B
 =y\vee\bigvee\bigl((A_1\dnw a)\setminus\set{b})\cup B\bigr)
 \]
while $((A_1\dnw a)\setminus\set{b})\cup B$ is contained in $A_1\setminus\set{b}$, we obtain the inequality\linebreak $p\leq y\vee\bigvee(A_1\setminus\set{b})$, which, as~$A_1\in\tcov(p)$, implies that $y=b$. This proves that~$A_1\dnw a$ joins to~$a$ tightly. Furthermore,
 \begin{align}
 \bigvee A_0&=\bigvee\famm{\bigvee(A_1\dnw a)}{a\in A_0}
 &&(\text{by the paragraph above})\notag\\
 &=\bigvee\bigcup\famm{A_1\dnw a}{a\in A_0}\notag\\
 &=\bigvee A_1&&(\text{because }A_1\leref A_0)\,.\tag*{\qed}
 \end{align}
\renewcommand{\qed}{}
\end{proof}

We shall use later the following immediate consequence of Lemma~\ref{L:SumRef}.

\begin{corollary}\label{C:SumRef}
Let $A\in\tcov(p)$. If $A\subseteq\J(L)$, then $A\in\mcov(p)$.
\end{corollary}

\begin{lemma}\label{L:AnTreeStruct}
Let $A_0,A_1\in\tcov(p)$ such that $A_1\leref A_0$. Then the sets $A_1\dnw x$, for $x\in A_0$, are pairwise disjoint.
\end{lemma}

\begin{proof}
Suppose that there are $x\in A_0$ and $z\in(A_1\dnw x)\cap B$ where
 \[
 B:=\bigcup\famm{A_1\dnw y}{y\in A_0\setminus\set{x}}\,.
 \]
 {}From $A_1\leref A_0$ it follows that $A_1=(A_1\dnw x)\cup B$.
 Furthermore, it follows from Lemma~\ref{L:SumRef} that the element
 $x':=\bigvee((A_1\dnw x)\setminus\set{z})$ (defined as a new zero
 element in case~$L$ has no zero and $A_1\dnw x=\set{z}$) is (strictly)
 smaller than~$x$ since, by Lemma~\ref{L:SumRef},
 $x=\bigvee A_1\dnw x$ tightly, while $\bigvee
 B=\bigvee(A_0\setminus\set{x})$. Now we compute
 \begin{align*}
 p&\leq\bigvee A_1\\
 &=\bigvee\bigl((A_1\dnw x)\cup B\bigr)\\
 &=\bigvee\bigl(((A_1\dnw x)\setminus\set{z})\cup B\bigr)
 &&(\text{because }z\in B)\\
 &=x'\vee\bigvee(A_0\setminus\set{x})\,,
 \end{align*}
which contradicts the assumption that~$A_0$ covers~$p$ tightly.
\end{proof}

\begin{lemma}\label{L:IncrCardCov}
Let $A_0,A_1\in\tcov(p)$ such that $A_1\leref A_0$. Then $|A_0\cap H|\leq|A_1\cap H|$ for each lower subset~$H$ of~$L$. In particular, $|A_0\dnw a|\leq|A_1\dnw a|$ for each $a\in L$.
\end{lemma}

\begin{proof}
For each $u\in A_0\cap H$, it follows from Lemma~\ref{L:SumRef} that $A_1\dnw u$ is nonempty; pick an element~$f(u)$ there. It follows from Lemma~\ref{L:AnTreeStruct} that~$f$ is one-to-one. As the range of~$f$ is contained in $A_1\cap H$, the first conclusion follows. The second conclusion is a particular case of the first one, with $H:=L\dnw a$.
\end{proof}

Although we shall not use the following result in the proof of Theorem~\ref{T:BiAlg2StrSp}, we record it for its independent interest.

\begin{proposition}[Interpolation property for tight covers]\label{P:Interp}
Let $A_0,A_1,A_2\in\tcov(p)$ such that $A_2\leref A_1\leref A_0$ and let $(a_0,a_2)\in A_0\times A_2$. If $a_2\leq a_0$, then there exists $a_1\in A_1$ such that $a_2\leq a_1\leq a_0$.
\end{proposition}

\begin{proof}
As $a_2\in A_2$ and $A_2\leref A_1$, there exists $a_1\in A_1\upw a_2$. Likewise, as $a_1\in A_1$ and $A_1\leref A_0$, there exists $a\in A_0\upw a_1$. Now~$a_2$ belongs to both sets~$A_2\dnw a_0$ and~$A_2\dnw a$, thus, by Lemma~\ref{L:AnTreeStruct}, $a=a_0$; and thus $a_1\leq a_0$.
\end{proof}

\begin{theorem}\label{T:BiAlg2StrSp}
Let~$p$ be an element in a complete, lower continuous lattice~$L$. If either~$p$ is countably compact or~$L$ is either well-founded or dually well-founded, then every join-cover of~$p$ can be refined to some minimal join-cover of~$p$.
\end{theorem}

\begin{proof}
Let~$C\in\cov(p)$, we wish to refine~$C$ to an element of~$\mcov(p)$. By Lemmas~\ref{L:MinCovDef} and~\ref{L:Irr2tight}, it suffices to prove that~$\tcov(p)$, endowed with the refinement order, is well-founded. Suppose otherwise. There exists a sequence $\vec A=\famm{A_n}{n<\omega}$ from~$\tcov(p)$ such that the inequality $A_{n+1}\lsref A_n$ holds for each $n<\omega$.

Say that an element~$x\in L$ is \emph{$\vec A$-reducible} if there exists a natural number~$k$ such that $|A_k\dnw x|\geq 2$. By Lemma~\ref{L:IncrCardCov}, the sequence $\famm{|A_k\dnw x|}{k<\omega}$ is nondecreasing, thus the $\vec A$-reducibility of~$x$ is equivalent to saying that $|A_k\dnw x|\geq 2$ for all large enough $k<\omega$. Now we set
 \[
 B_n:=\bigl\{x\in A_n\mid x\text{ is }\vec A\text{-reducible}\bigr\}\,,\quad
 \text{for each }n<\omega\,.
 \]
Observe that by Lemma~\ref{L:SumRef},
 \begin{equation}\label{Eq:AmminusBm}
 A_n\dnw x=\set{x}\text{ for all }m\leq n<\omega\text{ and all }x\in A_m\setminus B_m\,.
 \end{equation}
In particular, if~$B_n=\es$, then $A_n\subseteq A_{n+1}$, thus $A_n=A_{n+1}$ as each of these sets covers~$p$ tightly, which contradicts the assumption that $A_{n+1}\lsref A_n$. Therefore, $B_n$ is nonempty.

Furthermore, for each $n<\omega$, there exists $k>n$ such that $|A_k\dnw x|\geq2$ for each $x\in B_n$. An easy inductive argument yields a strictly increasing sequence $\famm{n_i}{i<\omega}$ of natural numbers, with $n_0=0$, such that for all $i<\omega$ and all $x\in B_{n_i}$, the set $A_{n_{i+1}}\dnw x$ has at least two elements. Set $\vec A':=\famm{A_{n_i}}{i<\omega}$. As, by Lemma~\ref{L:IncrCardCov}, the notions of $\vec A$-reducibility and $\vec A'$-reducibility are equivalent, we may replace~$\vec A$ by~$\vec A'$ and thus assume that $n_i=i$ for each $i<\omega$. Hence
 \begin{equation}\label{Eq:Bnniceincr}
 \text{For all }m<n<\omega\text{ and for all }x\in B_m\,,\ |A_n\dnw x|\geq2\,.
 \end{equation}

\begin{sclaim}
$B_n\cap A_{n+1}=\es$ and $B_{n+1}$ refines~$B_n$, for each $n<\omega$.
\end{sclaim}

\begin{scproof}
Let $u\in B_n\cap A_{n+1}$. It follows from Lemma~\ref{L:SumRef} that $A_{n+1}\dnw u$ covers~$u$ tightly, but $u\in A_{n+1}$, thus $A_{n+1}\dnw u=\set{u}$, which contradicts the assumption that $u\in B_n$ together with~\eqref{Eq:Bnniceincr}. Hence $B_n\cap A_{n+1}=\es$.

Now let $v\in B_{n+1}$. There exists $u\in A_n$ such that $v\leq u$. If~$u\notin B_n$, then, by~\eqref{Eq:AmminusBm}, $A_{n+1}\dnw u=\set{u}$, thus $v=u$, \contr\ as~$u$ is $\vec{A}$-irreducible while~$v$ is $\vec{A}$-reducible; so $u\in B_n$.
\end{scproof}

Now we consider the set~$T$ of all finite sequences $x=(x_0,x_1,\dots,x_n)$, where $n<\omega$ (the \emph{length} of~$x$), $x_i\in B_i$ for each $i\leq n$, and $x_{i+1}\leq x_i$ for each $i<n$. By our Claim, it follows that $x_{i+1}<x_i$ for each $i<n$. Furthermore, for each positive integer~$n$, we may pick $x_n\in B_n$ (because~$B_n\neq\es$), then, using our Claim, $x_{n-1}\in B_{n-1}$ such that $x_n\leq x_{n-1}$, and so on by induction; we get a finite sequence $(x_0,x_1,\dots,x_n)\in T$. As this can be done for every~$n$, the set $T$ is infinite.

As each~$A_n$ is finite, every element of~$T$ has only finitely many upper covers for the initial segment ordering. By K\"onig's Theorem, $T$ has an infinite branch, say $\famm{x_n}{n<\omega}$. As this branch is a (strictly) decreasing sequence, $L$ is not well-founded; thus, by assumption, either~$L$ is dually well-founded or~$p$ is countably compact.

Set $y_n:=\bigvee(A_n\setminus\set{x_n})$ for each $n<\omega$. As
$A_n\in\tcov(p)$, we get
\begin{equation}\label{Eq:pleqxnyn}
  p\leq x_n\vee y_n\text{ and }(\forall z<x_n)(p\nleq z\vee y_n)\,. 
\end{equation}
Furthermore, let $u\in A_n\setminus\set{x_n}$ and let $v\in
A_{n+1}\dnw u$. If $v=x_{n+1}$, then, by Lemma~\ref{L:AnTreeStruct},
$x_n=u$, \contr; hence $v\in A_{n+1}\setminus\set{x_{n+1}}$. By
joining over all the possible~$u$-s and~$v$-s, we obtain, using Lemma~\ref{L:SumRef}, that
\begin{align*}
  y_n & =\bigvee(A_n\setminus\set{x_n})= \bigvee\bigcup\famm{A_{n+1}\dnw
    u}{u\in A_n\setminus\set{x_n}} \\
  & \leq\bigvee(A_{n+1}\setminus\set{x_{n+1}})=y_{n+1}\,.
\end{align*}
Now set $x:=\bigwedge\famm{x_n}{n<\omega}$ (directed meet) and $y:=\bigvee\famm{y_n}{n<\omega}$ (directed join). {}From~\eqref{Eq:pleqxnyn} it follows that $p\leq x_n\vee y$ for each $n<\omega$, thus, as~$L$ is lower continuous, $p\leq x\vee y$.

If~$L$ is dually well-founded, then there exists $m<\omega$ such that
$y_m=y$, so $p\leq x\vee y_m$. If~$L$ is not dually well-founded,
then, by assumption, $p$ is countably compact, thus, as $p\leq x\vee
y=\bigvee_{n<\omega}(x\vee y_n)$ (directed join), there exists
$m<\omega$ such that $p\leq x\vee y_m$. Hence the latter conclusion
holds in every case, which, as $x\leq x_{m+1}<x_m$,
contradicts~\eqref{Eq:pleqxnyn}.
\end{proof}

\begin{corollary}\label{C:BiAlg2StrSp}
Let~$L$ be a lattice. If~$L$ is either well-founded or bi-algebraic, then it is strongly spatial.
\end{corollary}

\begin{proof}
A direct inductive argument (within~$L$) easily shows that if~$L$ is well-founded, then~$L$ is spatial. This conclusion also holds in case~$L$ is bi-algebraic, because then~$L$ is dually algebraic. Furthermore, if~$L$ is bi-algebraic, then, by Proposition~\ref{P:BasicCompGen}, every point of~$L$ is compact. The conclusion then follows from Theorem~\ref{T:BiAlg2StrSp}.
\end{proof}

Not every algebraic strongly spatial lattice is bi-algebraic. For example, the lattice $\Co(\omega)$ of all order-convex subsets of the chain~$\omega$ of all natural numbers is algebraic and atomistic, thus strongly spatial (cf. Proposition~\ref{P:AlgAt2StrSp}). On the other hand, it is not dually algebraic---in fact, by Wehrung~\cite[Corollary~12.5]{DLLB}, $\Co(\omega)$ cannot be embedded into any bi-algebraic lattice.

\section{{}From quasi-seeds to algebraic spatial lattices}\label{S:Seed2Spat}

The following easy result relates quasi-seeds to spatial lattices.

\begin{proposition}\label{P:QuasiSeedSp}
Let~$L$ be a compactly generated lattice. Then~$L$ is spatial if{f}~$\Jc(L)$ is a quasi-seed in~$L$.
\end{proposition}

\begin{proof}
If~$\Jc(L)$ is a quasi-seed in~$L$, then it is join-dense, thus~$L$ is spatial. Conversely, assume that~$L$ is spatial. Let $p\in\Jc(L)$ and let $X\in\cov(p)$, we must find $I\in\cov(p)$ contained in~$\Jc(L)$ such that $I\leref X$. As~$p$ is a point, it is compact (cf. Proposition~\ref{P:BasicCompGen}); as~$p\leq\bigvee X$ and each element of~$X$ is a join of points, we may assume that each element $x\in X$ is a finite join of points, so $x=\bigvee I_x$ where~$I_x$ is a finite subset of~$\Jc(L)$. It follows that the set $I:=\bigcup_{x\in X}I_x$ belongs to $\cov(p)$, is contained in~$\Jc(L)$, and refines~$X$.
\end{proof}

In the statement of the following lemma we make use of (covariant) \emph{Galois connections}, see for example Gierz \emph{et al.}~\cite{Comp}.

\begin{lemma}\label{L:VarSeed}
Let~$\Sigma$ be a subset in a \jzs~$L$. We define mappings $\eps\colon\Id\Sigma^\vee\to\Id L$ and $\pi\colon\Id L\to\Id\Sigma^\vee$ by the rules
 \begin{align}
 \eps(A)&:=L\dnw A\,,&&\text{for each }A\in\Id\Sigma^\vee\,,
 \label{Eq:DefnepsSig}\\
 \pi(B)&:=B\cap\Sigma^\vee\,,&&\text{for each }B\in\Id L\,.\label{Eq:DefnpiSig}
  \end{align}
Then the following statements hold:
\begin{enumerate}
\item The pair $(\eps,\pi)$ is a Galois connection between~$\Id\Sigma^\vee$ and $\Id L$.

\item $\pi\circ\eps$ is the identity on $\Id\Sigma^\vee$. In particular, $\pi$ is surjective.

\item The subset~$\Sigma$ is a pre-seed in~$L$ if{f}~$\pi$ is a lattice homomorphism.
\end{enumerate}  
\end{lemma}

\begin{proof}
(i). We must prove that $\eps(A)\subseteq B$ if{f} $A\subseteq\pi(B)$, for each $A\in\Id\Sigma^\vee$ and each $B\in\Id L$. This is trivial.

(ii). We must prove that $A=\Sigma^\vee\dnw A$, for each $A\in\Id\Sigma^\vee$. This is trivial.

(iii). Assume first that~$\Sigma$ is a pre-seed. It follows from~(i) and~(ii) that~$\pi$ is a surjective \mh. (The map~$\eps$ is also a one-to-one \jh, but we will not use this fact.) Hence it suffices to prove that $\pi(A\vee B)$ is contained in $\pi(A)\vee\pi(B)$, for all $A,B\in\Id L$. As~$\pi(A\vee B)$ is an ideal of~$\Sigma^\vee$, it is generated by its intersection with~$\Sigma$; thus it suffices to prove that each element $p\in\Sigma\cap\pi(A\vee B)$ belongs to $\pi(A)\vee\pi(B)$. By the definition of the map~$\pi$, the element~$p$ belongs to~$A\vee B$, that is, there exists $(a,b)\in A\times B$ such that $p\leq a\vee b$. As $\set{a,b}\in\cov(p)$ and~$\Sigma$ is a pre-seed, there exists $X\in\cov(p)$ contained in~$\Sigma$ such that $X\leref\set{a,b}$. The latter relation means that $X=(X\dnw a)\cup(X\dnw b)$. Setting $a':=\bigvee(X\dnw a)$ and $b':=\bigvee(X\dnw b)$, it follows that $p\leq\bigvee X=a'\vee b'$. As~$a'\in\pi(A)$ and~$b'\in\pi(B)$, the desired conclusion follows.

Conversely, assume that~$\pi$ is a lattice homomorphism. The assignment\linebreak $a\mapsto\pi(L\dnw a)=\Sigma^\vee\dnw a$ then defines a \jzh\ $\psi\colon L\to\Id\Sigma^\vee$. Let $p\in\Sigma$, let $n$ be a positive integer, and let $a_0,\dots,a_{n-1}\in L$ such that $p\leq\bigvee_{i<n}a_i$. Hence~$p$ belongs to $\psi\bigl(\bigvee_{i<n}a_i\bigr)=\bigvee_{i<n}\psi(a_i)$, that is, there are $x_i\in\psi(a_i)$, for $i<n$, such that $p\leq\bigvee_{i<n}x_i$. For each $i<n$, there exists a finite subset~$X_i$ of $\Sigma\dnw a_i$ such that $x_i=\bigvee X_i$. Therefore, the set $X:=\bigcup_{i<n}X_i$ is contained in~$\Sigma$, refines $\setm{a_i}{i<n}$, and $p\leq\bigvee_{i<n}x_i=\bigvee X$. Therefore, $\Sigma$ is a pre-seed.
\end{proof}

\begin{lemma}\label{L:QuasiSeed}
Let $L$ be a $0$-lattice with a quasi-seed $\Sigma$ and set~$S:=\Sigma^\vee$. Then $\Id S$ is an algebraic and spatial lattice, which contains a copy of~$L$ as a $0$-sublattice, and which belongs to the same lattice variety as~$L$. Furthermore, if~$\Sigma$ is a seed in~$L$, then~$\Id S$ is strongly spatial.
\end{lemma}

\begin{proof}
  Denote by $\eta\colon L\hookrightarrow\Id L$, $a\mapsto L\dnw a$ the
  canonical lattice embedding. As, by Lemma~\ref{L:VarSeed}, the
  natural map $\pi\colon\Id L\twoheadrightarrow\Id S$ is a lattice
  homomorphism, the composite $\psi:=\pi\circ\eta$ is also a lattice
  homomorphism; it is obviously zero-preserving. Note that
  $\psi(a)=S\dnw a$, for each $a\in L$. As~$\Sigma$ is join-dense
  in~$L$, it follows immediately that~$\psi$ is a $0$-lattice
  embedding from~$L$ into~$\Id S$. As~$\Id S$ is a homomorphic image
  of~$\Id L$ and the latter belongs to the lattice variety generated
  by~$L$ (cf. Gr\"atzer~\cite[Lemma~I.4.8]{GLT2}), so does~$\Id S$.

The lattice $\Id S$, being the ideal lattice of a \jzs, is an algebraic lattice. Now we claim that
 \begin{equation}\label{Eq:SdnwpPoint}
 S\dnw p\text{ is a point of }\Id S\,,\text{ with lower cover }S\ddnw p\,,
 \quad\text{for each }p\in\Sigma\,.
 \end{equation}
Indeed, as~$p$ is \jirr, the subset~$S\ddnw p$ is an ideal of~$S$; furthermore, every ideal of~$S$ properly contained in~$S\dnw p$ is contained in~$S\ddnw p$, which completes the proof of our claim.

Now let $A,B\in\Id S$ such that $A\not\subseteq B$ and let $a\in A\setminus B$. As~$a$ is a finite join of elements of~$\Sigma$, one of those elements belongs to~$A\setminus B$. This proves that the set~$P:=\setm{S\dnw p}{p\in\Sigma}$ is join-dense in~$\Id S$.
In particular, every point of~$\Id S$, being a join of elements of~$P$, belongs to~$P$. Therefore, $P$ is the set of all points of~$\Id S$. As~$P$ is join-dense in~$\Id S$, it follows that~$\Id S$ is spatial.

Finally assume that~$\Sigma$ is a seed in~$L$. We must prove that~$P$
is a seed in~$\Id S$. Let $p\in\Jc(L)$, let~$n$ be a positive integer,
and let~$A_0$, \dots, $A_{n-1}$ be ideals of~$S$ such that $S\dnw
p\subseteq\bigvee_{i<n}A_i$ (i.e., $p\in\bigvee_{i<n}A_i$) in~$\Id S$.
It follows that there are $a_0\in A_0$, \dots, $a_{n-1}\in A_{n-1}$
such that $p\leq\bigvee_{i<n}a_i$. As~$\Sigma$ is a seed in~$L$, there
exists $I\in\mcov(p)$ contained in~$\Sigma$ such that
$I\leref\set{a_0,\dots,a_{n-1}}$; the latter relation implies that
$\setm{S\dnw q}{q\in I}$ refines $\setm{A_i}{i<n}$. {}From
$p\leq\bigvee I$ it follows that
$S\dnw p\subseteq\bigvee_{q\in I}(S\dnw q)$. If this relation is not a minimal
join-covering, then, by~\eqref{Eq:SdnwpPoint} applied to $S\dnw
q$ for all $q\in I$, there exist $q\in I$ and $q'<q$
such that $p\leq q'\vee\bigvee(I\setminus\set{q})$, which contradicts $I\in\mcov(p)$. We have thus proved that $\setm{S\dnw q}{q\in I}$ is a minimal join-cover of~$S\dnw p$ in~$\Id S$ which refines $\setm{A_i}{i<n}$.
\end{proof}

\section{Seeds in $n$-distributive lattices}\label{S:Seedsndistr}

The following lemma is a slight improvement of Nation~\cite[Theorem~3.2]{Nation90}, and its proof is virtually the same.

\begin{lemma}\label{L:NatRefndistr}
Let~$n$ be a positive integer and let~$p$ be a \jirr\ element in a complete lower continuous $n$-distributive lattice~$L$. Then every element of~$\cov(p)$ can be refined to some element of~$\mcov(p)$.
\end{lemma}

\begin{proof}
As~$L$ is $n$-distributive and~$p$ is \jirr, every element of~$\icov(p)$ must have at most~$n$ elements. Now let $E\in\cov(p)$. By the above observation, there exists a maximal-sized irredundant join-cover~$P$ of~$p$ refining~$E$, and further, denoting by~$m$ the cardinality of~$P$, $m\leq n$. Furthermore, every $X\in\icov(p)$ refining~$P$ must have exactly~$m$ elements.

It follows from Lemma~\ref{L:Irr2tight} that there exists
$Q\in\tcov(p)$ such that $Q\leref P$. Every $R\in\icov(p)$
refining~$Q$ refines~$P$, so $|Q|=|R|=m$. Furthermore, for each $r\in
R$, there exists $f(r)\in Q$ such that $r\leq f(r)$, and $p\leq\bigvee
R\leq\bigvee_{r\in R}f(r)$. As the range of~$f$ is contained in~$Q$
and~$Q\in\icov(p)$, $f$ is surjective, and thus it is a bijection
from~$R$ onto~$Q$. Now $p\leq\bigvee R=\bigvee_{q\in Q}f^{-1}(q)$ with
$f^{-1}(q)\leq q$ for each~$q$. As $Q\in\tcov(p)$, $f^{-1}(q)=q$ for each~$q$, and therefore $Q=R$. By Lemma~\ref{L:MinCovDef}, it follows that $Q\in\mcov(p)$.
\end{proof}

\begin{corollary}\label{C:NatRefndistr}
Let~$n$ be a positive integer and let~$L$ be a dually algebraic $n$-distributive lattice. Then~$\J(L)$ and~$\Jc(L)$ are both seeds of~$L$.
\end{corollary}

\begin{proof}
As~$L$ is dually algebraic, it is spatial, thus~$\Jc(L)$ (and thus also~$\J(L)$) is join-dense in~$L$. Now let $p\in\J(L)$ and let~$X\in\cov(p)$. As~$L$ is dually algebraic, it is complete and lower continuous, thus it follows from Lemma~\ref{L:NatRefndistr} that there exists $I\in\mcov(p)$ such that $I\leref X$. By Lemma~\ref{L:MJCjirr}, $I$ is contained in~$\J(L)$. Furthermore, if~$p$ is a point, then it is compact (cf. Proposition~\ref{P:BasicCompGen}), thus, by Lemma~\ref{L:MJCjirr}, $I$ is contained in~$\Jc(L)$.
\end{proof}

\begin{theorem}\label{T:NatRefndistr}
Let~$n$ be a positive integer. Then every $n$-distributive lattice~$L$ can be embedded, within its variety, into an algebraic and strongly spatial lattice~$\ol{L}$ such that if~$L$ has a least element \pup{a largest element, both a least and a largest element, respectively}, then so does~$\ol{L}$.
\end{theorem}

\begin{proof}
The lattice~$L$ can be embedded, zero-preservingly in case~$L$ has a zero, into its filter lattice~$\Fil L$, \emph{via} the assignment $a\mapsto L\upw a$; furthermore, $\Fil L$ generates the same variety as~$L$ (cf. Section~I.3 and Lemma~I.4.8 in Gr\"atzer~\cite{GLT2}). As~$\Fil L$ is dually algebraic, we may assume that~$L$ is dually algebraic.

By Corollary~\ref{C:NatRefndistr}, $\J(L)$ is a seed in~$L$. By Lemma~\ref{L:QuasiSeed}, the ideal lattice~$\ol{L}$ of $\J(L)^\vee$ is algebraic and strongly spatial while it generates the same variety as~$L$. This takes care of everything except the preservation of the largest element of~$L$ if it exists. However, as every ideal of an algebraic (resp., strongly spatial) lattice is algebraic (resp., strongly spatial), the latter point is easily taken care of by replacing~$\ol{L}$ by the principal ideal generated by the largest element of~$L$.
\end{proof}

\section{Generation of the variety of all $n$-distributive lattices}\label{S:ndistrFinGen}

In Lemmas~\ref{L:Pjoinndistr} and~\ref{L:ApproxTerms}, we shall fix a positive integer~$n$ and a $n$-distributive, algebraic, and spatial lattice~$L$. The set~$\PP$ of all finite subsets of~$\Jc(L)$, partially ordered by inclusion, is directed. For each~$P\in\PP$, the \jz-subsemilattice~$P^\vee$ of~$L$ generated by~$P$ is a finite lattice, although not necessarily a sublattice of~$L$. We shall denote by~$\wedge_P$ the meet operation in~$P^\vee$. Likewise, we shall denote by~$\xt^P$ the interpretation of a lattice term~$\xt$ in the lattice~$P^\vee$.

\begin{lemma}\label{L:Pjoinndistr}
The lattice $P^\vee$ is $n$-distributive, for each $P\in\PP$.
\end{lemma}

\begin{proof}
By Nation~\cite[Theorem~3.1]{Nation90}, it suffices to prove that for every $p\in\J(P^\vee)$, every irredundant join-cover~$X$ of~$p$ in~$P^\vee$ contains at most~$n$ elements. Now observe that $\J(P^\vee)=P$. As~$P^\vee$ is a \jz-subsemilattice of~$L$, $X$ is also an irredundant join-cover of~$p$ in~$L$. As~$p\in\J(L)$ and~$L$ is $n$-distributive, $|X|\leq n$.
\end{proof}

Denote by $\proj{a}{P}$ the largest element of~$P^\vee$ below~$a$, for each $a\in L$ and each $P\in\PP$. For a vector $\vec a=(a_1,\dots,a_k)$, we shall use the notation $\proj{\vec a}{P}:=(\proj{(a_1)}{P},\dots,\proj{(a_k)}{P})$.

\begin{lemma}\label{L:ApproxTerms}
The following statements hold, for every lattice term~$\xt$ with $k$ variables and every vector $\vec a=(a_1,\dots,a_k)\in L^k$:
\begin{enumerate}
\item $\xt^P(\proj{\vec a}{P})\leq\xt^Q(\proj{\vec a}{Q})$ holds for all $P\subseteq Q$ in~$\PP$;

\item $\xt(\vec a)=\bigvee\famm{\xt^P(\proj{\vec a}{P})}{P\in\PP}$ in~$L$.
\end{enumerate}
\end{lemma}

We shall abbreviate the conjunction of~(i) and~(ii) above by the notation
 \begin{equation}\label{Eq:NotIncrJoin}
 \xt(\vec a)=\JJd\famm{\xt^P(\proj{\vec a}{P})}{P\in\PP}\,.
 \end{equation}

\begin{proof}
  (i). We argue by induction on the complexity of~$\xt$. The result is
  obvious in case~$\xt$ is a projection. Now, supposing we have proved the
  statement for terms~$\xt_0$ and~$\xt_1$, we
  must prove it for both terms $\xt_0\wedge\xt_1$ and
  $\xt_0\vee\xt_1$. Let $\vec a\in L^k$, set
  $b_i^R:=\xt_i^R(\proj{\vec a}{R})$ for $i\in\set{0,1}$ and
  $R\in\set{P,Q}$. It follows from the induction hypothesis that
  $b_i^P\leq b_i^Q$ for all $i\in\set{0,1}$. Now $b_0^P\wedge_Pb_1^P$
  is an element of~$P^\vee$, thus of~$Q^\vee$, contained in both
  elements~$b_0^Q$ and~$b_1^Q$, thus also in $b_0^Q\wedge_Qb_1^Q$;
  this proves that
 \[
 (\xt_0\wedge\xt_1)^P(\proj{\vec a}{P})=b_0^P\wedge_Pb_1^P\leq b_0^Q\wedge_Qb_1^Q
 =(\xt_0\wedge\xt_1)^Q(\proj{\vec a}{Q})\,.
 \]
Next, $b_i^P\leq b_i^Q\leq b_0^Q\vee b_1^Q$ for each $i\in\set{0,1}$, thus $b_0^P\vee b_1^P\leq b_0^Q\vee b_1^Q$; this proves that
 \[
 (\xt_0\vee\xt_1)^P(\proj{\vec a}{P})=b_0^P\vee b_1^P\leq b_0^Q\vee b_1^Q
 =(\xt_0\vee\xt_1)^Q(\proj{\vec a}{Q})\,.
 \]

(ii). Again, we argue on the complexity of~$\xt$. In case~$\xt$ is a projection, we must prove the relation
 \begin{equation}\label{Eq:ajoinofaP}
 a=\bigvee\famm{\proj{a}{P}}{P\in\PP}\,,\quad\text{for each }a\in L\,. 
 \end{equation}
Trivially $a\geq\proj{a}{P}$ for each~$P$. Conversely, let $b\in L$ such that $\proj{a}{P}\leq b$ for each $P\in\PP$, we must prove that $a\leq b$. As~$\Jc(L)$ is join-dense in~$L$, it suffices to prove that $p\leq b$ for each point~$p$ such that $p\leq a$. Setting $P:=\set{p}$, we obtain that $p=\proj{a}{P}\leq b$. This completes the proof of~\eqref{Eq:ajoinofaP}.

For the induction step, let~$\xt_0$ and~$\xt_1$ be lattice terms for which the induction hypothesis holds at a vector $\vec a\in L^k$, that is,
 \begin{equation}\label{Eq:IHfortiveca}
 \xt_i(\vec a)=\JJd\famm{\xt_i^P(\proj{\vec a}{P})}{P\in\PP}\,,\quad\text{for each }i\in\set{0,1}\,.
 \end{equation}
We must prove the inequality
 \begin{equation}\label{Eq:tveca}
 \xt(\vec a)\leq\JJd\famm{\xt^P(\proj{\vec a}{P})}{P\in\PP}\,,
 \end{equation}
for each $\xt\in\set{\xt_0\wedge\xt_1,\xt_0\vee\xt_1}$. We first deal with the meet. Let~$p\in\Jc(L)$ lying below $\xt_0(\vec a)\wedge\xt_1(\vec a)$. For each $i\in\set{0,1}$, as~$p$ is a compact element lying below the right hand side of~\eqref{Eq:IHfortiveca}, which is a directed join by~(i), there exists $P_i\in\PP$ such that $p\leq\xt_i^{P_i}(\proj{\vec a}{P_i})$. Set $P:=P_0\cup P_1\cup\set{p}$. It follows from~(i) that $p\leq\xt_i^{P_i}(\proj{\vec a}{P_i})\leq\xt_i^{P}(\proj{\vec a}{P})$, for each  $i\in\set{0,1}$; and thus, as $p\in P$, we get $p\leq(\xt_0\wedge\xt_1)^P(\proj{\vec a}{P})$.

Now suppose that $\xt=\xt_0\vee\xt_1$. For each $i\in\set{0,1}$,
 \[
 \xt_i(\vec a)=\JJd\famm{\xt_i^P(\proj{\vec a}{P})}{P\in\PP}\leq
 \JJd\famm{\xt^P(\proj{\vec a}{P})}{P\in\PP}\,,
 \]
thus, by forming the join of those inequalities for $i\in\set{0,1}$, we obtain~\eqref{Eq:tveca}.
\end{proof}

We do not claim that the lattice~$P^\vee$ satisfies every lattice
identity satisfied by~$L$ (in fact, easy examples show that this is
not the case as a rule). However, we can still prove the following
result.

\begin{theorem}\label{T:Finn-distr}
Let~$n$ be a positive integer. Then the variety~$\cD_n$ of all $n$-distributive lattices is generated by its finite members. Consequently, the word problem for free lattices in~$\cD_n$ is decidable.
\end{theorem}

\begin{proof}
Let~$\xs$ and~$\xt$ be lattice terms, say of arity~$k$, such that every finite lattice in~$\cD_n$ satisfies the identity~$\xs=\xt$. We must prove that every lattice in~$\cD_n$ satisfies that identity. By Theorem~\ref{T:NatRefndistr}, it suffices to prove that every $n$-distributive, algebraic, and spatial lattice~$L$ satisfies the identity. Define the directed poset~$\PP$ as above. It follows from Lemma~\ref{L:Pjoinndistr} that~$P^\vee$ is $n$-distributive, for each $P\in\PP$. Thus, by assumption, $P^\vee$ satisfies the identity~$\xs=\xt$. Therefore, for each vector $\vec a\in L^k$,
 \begin{align*}
 \xs(\vec a)&=\JJd\famm{\xs^P(\proj{\vec a}{P})}{P\in\PP}
 &&(\text{by Lemma~\ref{L:ApproxTerms}})\\
 &=\JJd\famm{\xt^P(\proj{\vec a}{P})}{P\in\PP}
 &&(\text{because }P^\vee\text{ satisfies }\xs=\xt)\\
 &=\xt(\vec a)
 &&(\text{by Lemma~\ref{L:ApproxTerms}})\,,
 \end{align*}
which concludes the proof of the first statement. The decidability statement then follows from McKinsey~\cite[Theorem~3]{McKi}.
\end{proof}

\section{Seeds in modular lattices}\label{S:Mod}

\begin{lemma}\label{L:MinCovMod}
Let $p$, $\ol{p}$, and~$a$ be elements in a modular lattice~$L$ with~$p$ and~$\ol{p}$ both \jirr, and let $Q\in\mcov(a)$ satisfying the following conditions:
\begin{enumerate}
\item $p\leq\ol{p}\vee a$ is a tight join-cover;

\item $\ol{p}\notin Q$;

\item $p\leq\bigvee(\set{\ol{p}}\cup Q)$ irredundantly.
\end{enumerate}
Then $\set{\ol{p}}\cup Q$ belongs to~$\mcov(p)$.
\end{lemma}

\begin{proof}
Set $b:=\bigvee Q$.
We first prove the following statement:
 \begin{equation}\label{Eq:pnleqp'veea}
 p\nleq x\vee b\text{ for each }x<\ol{p}\,.
 \end{equation}
 Suppose that $p\leq x\vee b$. As $x<\ol{p}$ and $a\leq b$, it follows
 from  Zassenhaus' Butterfly Lemma (cf. Gr\"atzer
 \cite[Theorem~IV.1.13]{GLT2}) that the sublattice of~$L$ generated by $\set{x,\ol{p},a,b}$ is distributive. Hence, setting
 $u:=\ol{p}\wedge(x\vee b)=x\vee(\ol{p}\wedge b)$, we get
 $p\leq(\ol{p}\vee a)\wedge(x\vee b)=u\vee a$, with
 $u\leq \ol{p}$. By the minimality
 assumption on~$\ol{p}$ in~(i), it follows that $u=\ol{p}$, thus,
 as~$\ol{p}$ is \jirr\ and $x<\ol{p}$, we get $\ol{p}\leq b$. As
 $a\leq b$, it follows from~(i) that $p\leq\ol{p}\vee a\leq b=\bigvee
 Q$, thus, by~(iii), $\ol{p}\in Q$, which contradicts~(ii). This
 completes the proof of~\eqref{Eq:pnleqp'veea}.
 
 Now let $q\in Q$, set $Q':=Q\setminus\set{q}$ and
 $a':=\bigvee Q'$ (we may set it equal to a new zero
 element of~$L$ in case $Q'=\es$ and~$L$ has no zero), we must prove
 the statement
 \begin{equation}\label{Eq:pnleqp0veey}
 p\nleq\ol{p}\vee y\vee a'\text{ for each }y<q\,.
 \end{equation}
Suppose that $p\leq\ol{p}\vee y\vee a'$. As $p\leq\ol{p}\vee a$, it follows from the modularity of~$L$ that $p\leq(\ol{p}\vee a)\wedge(\ol{p}\vee y\vee a')=\ol{p}\vee(a\wedge(\ol{p}\vee y\vee a'))$. As $a\wedge(\ol{p}\vee y\vee a')\leq a$, it follows from the minimality assumption on~$a$ in~(i) that $a\leq\ol{p}\vee y\vee a'$. As $a\leq\bigvee Q=q\vee a'$, it follows from the modularity of~$L$ that
 \[
 a\leq(q\vee a')\wedge(\ol{p}\vee y\vee a')=
 \bigl(q\wedge(\ol{p}\vee y\vee a')\bigr)\vee a'
 =\bigl(q\wedge(\ol{p}\vee y\vee a')\bigr)\vee\bigvee Q'\,,
 \]
 with $q\wedge(\ol{p}\vee y\vee a')\leq q$. Hence, as~$Q\in\mcov(a)$,
 we get that $q\leq\ol{p}\vee y\vee a'$. As $y<q$, it follows from the
 modularity of~$L$ that $q=q\wedge(\ol{p}\vee y\vee
 a')=y\vee(q\wedge(\ol{p}\vee a'))$. As~$q$ is \jirr\ and $y<q$, it
 follows that $q\leq\ol{p}\vee a'$, thus $a\leq\bigvee Q=q\vee
 a'\leq\ol{p}\vee a'$, and thus $p\leq\ol{p}\vee a\leq\ol{p}\vee
 a'=\ol{p}\vee\bigvee Q'$, which contradicts the combination of~(ii)
 and~(iii). This completes the proof of~\eqref{Eq:pnleqp0veey}.
 
 The combination of~\eqref{Eq:pnleqp'veea} and~\eqref{Eq:pnleqp0veey}
 (the latter being stated for each $q\in Q$) means that
 $\set{\ol{p}}\cup Q\in\tcov(p)$. As moreover $\set{\ol{p}}\cup Q\subseteq\J(L)$,
 it follows from Corollary~\ref{C:SumRef} that $\set{\ol{p}}\cup Q\in\mcov(p)$.
\end{proof}

As in Herrmann, Pickering, and Roddy~\cite{HPR}, the \emph{collinearity relation} in a modular lattice~$L$ is defined on the \jirr\ elements of~$L$ by letting $\col(p,q,r)$ hold if~$p$, $q$, $r$ are pairwise incomparable and $p\vee q=p\vee r=q\vee r$.

\begin{lemma}\label{L:Col2Min}
Let $p$, $q$, $r$ be \jirr\ elements in a modular lattice~$L$. Then $\col(p,q,r)$ implies that $p\leq q\vee r$ is a tight cover.
\end{lemma}

\begin{proof}
  Let, say, $p\leq x\vee r$ with $x<q$. As $q\leq p\vee r$, it
  follows that $q\leq x\vee r$. Using $x<q$ and the modularity
  of~$L$, we obtain $q=x\vee(q\wedge r)$, and thus, by the \jirry\ of~$q$, we
  get $q\leq r$, \contr.
\end{proof}

\begin{lemma}\label{L:Irr2Min}
Let $n$ be a positive integer and let $p$, $a_1$, \dots, $a_n$ be elements in a modular spatial lattice~$L$ with~$p$ a point and $p\leq a_1\vee\cdots\vee a_n$ irredundantly. Then there are points $q_1\leq a_1$, \dots, $q_n\leq a_n$ such that $\set{q_1,\dots,q_n}$ is a minimal cover of~$p$.
\end{lemma}

\begin{proof}
We argue by induction on~$n$. For $n=1$ the result is trivial. For $n=2$, it follows from Herrmann, Pickering, and Roddy~\cite[Lemma~2.2]{HPR} that there are points $q_1\leq a_1$ and $q_2\leq a_2$ such that $\col(p,q_1,q_2)$ holds. By Lemma~\ref{L:Col2Min}, it follows that $\set{q_1,q_2}\in\mcov(p)$.

Now suppose that the result holds for~$n$ and suppose that
 \begin{equation}\label{Eq:pleqa02anirr}
 p\leq a_0\vee a_1\vee\cdots\vee a_n\text{ irredundantly}\,,
 \end{equation}
with~$p$ a point. By the case $n=2$, there are points $q_0\leq a_0$ and $q\leq a_1\vee\cdots\vee a_n$ such that $\set{q_0,q}\in\mcov(p)$. If $q\leq\bigvee_{i\in[1,n]\setminus\set{j}}a_i$ for some $j\in[1,n]$, then, as $p\leq q_0\vee q\leq a_0\vee q$, we obtain that $p\leq\bigvee_{i\in[0,n]\setminus\set{j}}a_i$, \contr; hence we obtain that
 \begin{equation}\label{Eq:CovqIrred}
 q\leq a_1\vee\cdots\vee a_n\text{ irredundantly}\,.
 \end{equation}
By the induction hypothesis, there are points $q_1\leq a_1$, \dots, $q_n\leq a_n$ such that the set $Q:=\set{q_1,\dots,q_n}$ belongs to~$\mcov(q)$. If $p\leq\bigvee Q$, then $p\leq\bigvee_{1\leq i\leq n}a_i$, which contradicts~\eqref{Eq:pleqa02anirr}; hence $p\nleq\bigvee Q$, but $p\leq q_0\vee q\leq q_0\vee\bigvee Q$, and thus $q_0\notin Q$. If $p\leq q_0\vee\bigvee(Q\setminus\set{q_j})$ for some $j\in[1,n]$, then $p\leq\bigvee_{i\in[0,n]\setminus\set{j}}a_i$, \contr; hence $p\leq\bigvee(\set{q_0}\cup Q)$ irredundantly. Now it follows from Lemma~\ref{L:MinCovMod} that $\set{q_0}\cup Q=\set{q_0,q_1,\dots,q_n}$ belongs to~$\mcov(p)$.
\end{proof}

\begin{theorem}\label{T:ModSp2StrSp}
Every modular spatial lattice is strongly spatial.
\end{theorem}

\begin{proof}
Let $p$ be a point in a modular spatial lattice~$L$ and let $X\in\cov(p)$. There exists a subset~$Y$ of~$X$ such that $p\leq\bigvee Y$ irredundantly. By Lemma~\ref{L:Irr2Min}, $Y$ can be refined to an element of~$\mcov(p)$.
\end{proof}

As Herrmann, Pickering, and Roddy prove in~\cite{HPR} that every modular lattice can be embedded, within its variety and zero-preservingly in case there is a zero, into some algebraic and spatial lattice, we thus obtain the following.

\begin{corollary}\label{C:ModSp2StrSp}
Every modular lattice~$L$ can be embedded, within its variety, into some algebraic strongly spatial lattice. If~$L$ has a zero, then the embedding can be taken zero-preserving.
\end{corollary}

\section{A class of lattices with a \bdl\ parameter}
\label{S:KL(D)}

In this section we shall prepare the ground for the construction of a lattice without any variety-preserving extension to an algebraic and spatial one, cf. Theorem~\ref{T:NoVarEmb}.

\begin{notation}\label{Not:KL(D)}
Let~$D$ be a nontrivial \bdl. We set
 \[
 x^-:=\begin{cases}
 0\,,&\text{if }x<1\\
 1\,,&\text{if }x=1
 \end{cases}\,,
 \quad\text{for each }x\in D\,.
 \]
Furthermore, we set
 \begin{align*}
 \rL(D)&:=\two\times D\times\two\times\two\,,\\
 \rK(D)&:=\setm{(x_0,x_1,x_2,x_3)\in\rL(D)}
 {(\forall i<j<k\text{ in }\set{0,1,2,3})(x_i^-\wedge x_k^-\leq x_j)}\,,
 \end{align*}
and we define a map $\gamma\colon\rL(D)\to\rL(D)$ by setting
 \[
 \gamma(x)_j:=x_j\vee\bigvee\famm{x_i^-\wedge x_k^-}{i<j<k}\,,
 \quad\text{for each }j\in\set{0,1,2,3}\,.
 \]
We shall denote by $\veec$ the join operation in~$\rL(D)$. Hence
 \begin{equation}\label{Eq:veec}
 (x\veec y)_i=x_i\vee y_i\,,\quad\text{for all }x,y\in\rL(D)\text{ and all }i<4\,.
 \end{equation}
We shall denote by~$q_i$ the element of~$\rL(D)$ with~$1$ at the $i$-th place and~$0$ elsewhere, for each $i\in\set{0,1,2,3}$. Furthermore, we set $xq_1:=(0,x,0,0)$, for each $x\in D$. Obviously $0q_1=0$, $1q_1=q_1$, and every element either of the form~$q_i$ or of the form $xq_1$ belongs to~$\rK(D)$.
\end{notation}

In Lemma~\ref{L:gammclos} and Corollary~\ref{C:gammclos} we shall fix a nontrivial \bdl~$D$.

\begin{lemma}\label{L:gammclos}
The element $\gamma(x)$ is the least element of~$\rK(D)$ above~$x$, for each $x\in\rL(D)$.
\end{lemma}

\begin{proof}
It is obvious that~$\gamma$ is isotone, that $\gamma(x)\geq x$ for each $x\in\rL(D)$, and that $\gamma(x)=x$ if{f} $x\in\rK(D)$. Hence it suffices to prove that $y:=\gamma(x)$ belongs to~$\rK(D)$, for each $x\in\rL(D)$. Let $i<j<k$ in $\set{0,1,2,3}$, we must prove that if $y_i=y_k=1$, then $y_j=1$. {}From $y_i=1$ it follows that there exists $i'\leq i$ such that $x_{i'}=1$. Similarly, from $y_k=1$ it follows that there exists $k'\geq k$ such that $x_{k'}=1$. {}From $i'<j<k'$ it follows that $y_j\geq x_{i'}^-\wedge x_{k'}^-=1$.
\end{proof}

As $\rL(D)$, endowed with the componentwise ordering, is a bounded lattice, we obtain the following.

\begin{corollary}\label{C:gammclos}
The set~$\rK(D)$, endowed with the componentwise ordering, is a closure system in~$\rL(D)$. In particular, it is a bounded lattice, and also a meet-sub\-semi\-lat\-tice of~$\rL(D)$.
\end{corollary}

\begin{lemma}\label{L:K(B)SI}
The following statements hold, for any nontrivial Boolean lattice~$B$.
\begin{enumerate}
\item The equality $\con_{\rK(B)}(0,q_2)=\con_{\rK(B)}(0,xq_1)$ holds for each $x\in B\setminus\set{0}$.

\item The lattice~$\rK(B)$ is subdirectly irreducible, with minimal nonzero congruence $\con_{\rK(B)}(0,q_1)$.
\end{enumerate}
\end{lemma}

\begin{proof}
(i). Let~$\theta$ be a congruence of~$\rK(B)$. Suppose first that $q_2\equiv_{\theta}0$. {}From $q_1\leq q_0\vee q_2$ it follows that $q_1\leq_{\theta}q_0$, but $q_1\wedge q_0=0$, thus $q_1\equiv_{\theta}0$. This implies in turn that $xq_1\equiv_{\theta}0$ for each $x\in B$. Finally suppose that $xq_1\equiv_{\theta}0$ for some $x\in B\setminus\set{0}$. Denote by~$y$ the complement of~$x$ in~$B$. {}From $q_2\leq q_1\vee q_3=xq_1\vee yq_1\vee q_3$ it follows that $q_2\leq_{\theta}yq_1\vee q_3$. {}From $y<1$ it follows that $yq_1\vee q_3=(0,y,0,1)$, thus $q_2\wedge(yq_1\vee q_3)=0$, and thus $q_2\equiv_{\theta}0$.

(ii). It suffices to prove that $q_1\equiv_{\theta}0$ for each nonzero congruence~$\theta$ of~$\rK(B)$. There are $x<y$ in~$\rK(B)$ such that $x\equiv_{\theta}y$. As $z=\bigvee_{i<4}(z\wedge q_i)$ for each $z\in\rK(B)$ we may assume that $y\leq q_i$ for some $i<4$. If $i=2$ then $x=0$ and $y=q_2$, thus $q_2\equiv_{\theta}0$, and thus, by~(i), $q_1\equiv_{\theta}0$ and we are done. If $i=1$, then, as~$B$ is Boolean, there exists $z\in\rK(B)\dnw q_1$ such that $x\vee z=y$ while $x\wedge z=0$; observe that $0<z\leq q_1$ and $z\equiv_{\theta}0$. It follows from~(i) above that $q_1\equiv_{\theta}0$. If $i=0$ then $q_0\equiv_{\theta}0$, but $q_1\leq q_0\vee q_2$, thus $q_1\leq_{\theta}q_2$, but $q_1\wedge q_2=0$, thus $q_1\equiv_{\theta}0$. If $i=3$ then~$q_3\equiv_{\theta}0$, thus, as $q_2\leq q_1\vee q_3$, we get $q_2\leq_{\theta}q_1$, but $q_1\wedge q_2=0$, thus $q_2\equiv_{\theta}0$, and thus, by~(i) above, $q_1\equiv_{\theta}0$.
\end{proof}

\begin{lemma}\label{L:LocFin}
The lattice~$\rK(D)$ is locally finite, for each nontrivial \bdl~$D$. In fact, $\rK(D)$ generates a locally finite lattice variety.
\end{lemma}

\begin{proof}
Let $n$ be a positive integer and let $x_1,\dots,x_n\in\rK(D)$. Denote by~$C$ the $0,1$-sublattice of~$D$ generated by $\set{(x_1)_1,\dots,(x_n)_1}$. Then the sublattice of~$\rK(D)$ generated by $\set{x_1,\dots,x_n}$ is contained in~$\rK(C)$, which has at most $2^{2^n+3}$ elements. As this bound depends of~$n$ only, the local finiteness statement about the variety follows easily by using a standard argument of universal algebra.
\end{proof}

\section{Local distributivity of lattices of the form $\rK(D)$}\label{S:LocDistr}

Throughout this section we shall denote by~$\cK$ the class of all lattices of the form $\rK(B)$, for nontrivial Boolean lattices~$B$, and by~$\cV$ the lattice variety generated by~$\cK$. Furthermore, we denote by~$B_0$ the (Boolean) lattice of all subsets of~$\omega$ that are either finite or cofinite. It is in fact easy to verify that~$\cV$ is generated by $\rK(B_0)$.

\begin{lemma}\label{L:NonDistrId}
Let~$D$ be a nontrivial \bdl\ and let $a\in\rK(D)$. If $\rK(D)\dnw a$ is not distributive, then~$a$ is either equal to $(1,1,1,0)$, or to $(0,1,1,1)$, or to $(1,1,1,1)$.
\end{lemma}

\begin{proof}
  There are $x,y,z\in\rK(D)\dnw a$ such that $(x\vee y)\wedge
  z>(x\wedge z)\vee(y\wedge z)$. Set $t:=x\veec y$; observe that
  $t\leq a$. It follows from the distributivity of~$D$ that $x\vee
  y>t$ (cf.~\eqref{Eq:veec}), that is, $(x\vee y)_j>t_j$ for some
  $j<4$. By Lemma~\ref{L:gammclos}, $x\vee y=\gamma(t)$. Hence, by the
  definition of~$\gamma$, there are $i<j$ and $k>j$ such that
  $t_i=t_k=1$. {}From $t\leq a$ it follows that $a_i=a_k=1$.
  As~$a\in\rK(D)$, it follows that~$a$ takes the value~$1$ on at least
  three consecutive elements of $\set{0,1,2,3}$. The conclusion
  follows.
\end{proof}

\begin{corollary}\label{C:NonDistrIdSI}
Every subdirectly irreducible member of~$\cV$ satisfies the sentence
 \begin{equation}\label{Eq:1stOrdSent}
 (\forall\vx,\vy,\vz,\vt_1,\vt_2)\bigl(\vx\vee\vy\vee\vz<\vt_1<\vt_2
 \ \Longrightarrow\ (\vx\vee\vy)\wedge\vz=
 (\vx\wedge\vz)\vee(\vy\wedge\vz)\bigr)\,.
 \end{equation}
\end{corollary}

\begin{proof}
  As both $(1,1,1,0)$ and $(0,1,1,1)$ are lower covers of $(1,1,1,1)$
  in~$\rK(D)$, it follows from Lemma~\ref{L:NonDistrId} that every
  member of~$\cK$ satisfies~\eqref{Eq:1stOrdSent}. By \L os'
  Ultraproduct Theorem, every ultraproduct of members of~$\cK$ also
  satisfies~\eqref{Eq:1stOrdSent}, and thus every sublattice of such
  an ultraproduct satisfies~\eqref{Eq:1stOrdSent}. By J\'onsson's
  Lemma, for every subdirectly irreducible member~$L$ of~$\cV$, there
  are a sublattice~$\ol{L}$ of an ultraproduct of members of~$\cK$ and
  a surjective lattice homomorphism $f\colon\ol{L}\twoheadrightarrow
  L$. Now let $x,y,z,t_1,t_2\in L$ such that $x\vee y\vee z<t_1<t_2$,
  with respective preimages
  $\ol{x},\ol{y},\ol{z},\ol{t}_1,\ol{t}_2\in\ol{L}$ under~$f$.
  Necessarily,
  \[
  \ol{x}\vee\ol{y}\vee\ol{z}<\ol{x}\vee\ol{y}\vee\ol{z}\vee\ol{t}_1<
  \ol{x}\vee\ol{y}\vee\ol{z}\vee\ol{t}_1\vee\ol{t}_2\,,
  \]
  thus, as~$\ol{L}$
  satisfies~\eqref{Eq:1stOrdSent},
  $(\ol{x}\vee\ol{y})\wedge\ol{z}=(\ol{x}\wedge\ol{z})\vee(\ol{y}\wedge\ol{z})$.
  By applying the homomorphism~$f$, we obtain that $(x\vee y)\wedge
  z=(x\wedge y)\vee(x\wedge z)$.
\end{proof}

\begin{corollary}\label{C:LocSIDistr}
Let~$D$ be a nontrivial \bdl\ and let $L$ be a subdirectly irreducible member of~$\cV$ containing~$\rK(D)$ as a sublattice. Then $L\dnw q_1$ is distributive.
\end{corollary}

\begin{proof}
This follows immediately from the inequalities $q_1<q_1\vee q_2<q_1\vee q_2\vee q_3$ (in~$\rK(D)$, thus in~$L$) together with Corollary~\ref{C:NonDistrIdSI}.
\end{proof}

\begin{lemma}\label{L:2/3distr}
Let~$B$ be a nontrivial \bl\ and let~$L$ be a lattice in~$\cV$ containing~$\rK(B)$ as a sublattice. Then $(x\vee y)\wedge z=(x\wedge z)\vee(y\wedge z)$ holds for all $x,y\in\rK(B)\dnw q_1$ and all $z\in L$.
\end{lemma}

\begin{proof}
By replacing~$z$ by $z\wedge q_1$, we may assume that~$z\leq q_1$. It follows from Birkhoff's Subdirect Representation Theorem that~$L$ is a subdirect product of subdirectly irreducible lattices. Hence it suffices to prove that the relation
 \begin{equation}\label{Eq:DistrModpsi}
 (x\vee y)\wedge z\equiv_\psi(x\wedge z)\vee(y\wedge z)
 \end{equation}
holds for each completely \mirr\ congruence~$\psi$ of~$L$. Set $L_\psi:=L/{\psi}$ and denote by $p_\psi\colon L\twoheadrightarrow L_\psi$ the canonical projection. We separate cases.
\begin{itemize}
\item[\textbf{Case 1}.] $p_\psi\res_{\rK(B)}$ is one-to-one. It follows from Corollary~\ref{C:LocSIDistr} that $L_\psi\dnw(q_1/{\psi})$ is distributive. The relation~\eqref{Eq:DistrModpsi} follows.

\item[\textbf{Case 2}.] $p_\psi\res_{\rK(B)}$ is not one-to-one. It follows from Lemma~\ref{L:K(B)SI} that $q_1\equiv_{\psi}0_{\rK(B)}$, thus, as $x,y\in\rK(B)\dnw q_1$, we get $x\equiv_{\psi}y\equiv_{\psi}0_{\rK(B)}$. The relation~\eqref{Eq:DistrModpsi} follows trivially.
\end{itemize}
This concludes the proof.
\end{proof}

\begin{theorem}\label{T:NoVarEmb}
There is no algebraic spatial lattice~$L$ in~$\cV$ that contains~$\rK(B_0)$ as a sublattice.
\end{theorem}

\begin{proof}
  Set $K:=\rK(B_0)$, $p_n:=(0,\set{n},0,0)$, and
  $u_n:=(0,\omega\setminus\set{n},0,0)$ for each $n<\omega$. Observe
  that~$p_n$ and~$u_n$ both belong to~$K$ and
  $p_n\vee u_n=q_1$ in~$K$. As~$L$ is spatial, there exists a
  point~$x$ of~$L$ such that $x\leq q_2$ and $x\nleq 0_K$. As $x\leq
  q_2\leq q_1\vee q_3$ and~$x$ is compact while~$L$ is spatial, there
  are a positive integer~$m$ and points~$y_1,\dots,y_m\leq q_1$ of~$L$
  such that
  \begin{equation}\label{Eq:Ineqonx}
    x\leq y_1\vee\cdots\vee y_m\vee q_3\,.
  \end{equation}
  Set $I:=\setm{i\in\set{1,\dots,m}}{(\exists n<\omega)(y_i\leq
    p_n)}$. We may assume that $I=\set{1,\dots,r}$ for some
  $r\in\set{0,\dots,m}$. There exists~$\ell<\omega$ such that
  $y_1\vee\cdots\vee y_r\leq p_0\vee\cdots\vee p_{\ell-1}$; it follows
  that
  \begin{equation}\label{Eq:Ineqwithell}
    y_1\vee\cdots\vee y_r\leq u_\ell\,.
  \end{equation}
  As $y_i\leq q_1=p_{\ell}\vee u_{\ell}$, for each
  $i\in\set{1,\dots,m}$, it follows from Lemma~\ref{L:2/3distr} that
  $y_i=(y_i\wedge p_{\ell})\vee(y_i\wedge u_{\ell})$, and thus, by the
  \jirry\ of~$y_i$, either $y_i\leq p_{\ell}$ or $y_i\leq u_{\ell}$.
  It follows that $y_i\leq u_{\ell}$ for each $i>r$, and so
 \begin{equation}\label{Eq:othersyi}
 y_{r+1}\vee\cdots\vee y_m\leq u_\ell\,.
 \end{equation}
By putting~\eqref{Eq:Ineqwithell} and~\eqref{Eq:othersyi} together, we obtain the inequality
 \begin{equation}\label{Eq:allyi}
 y_1\vee\cdots\vee y_m\leq u_\ell\,.
 \end{equation}
By putting~\eqref{Eq:Ineqonx} and~\eqref{Eq:allyi} together, we obtain
 \[
 x\leq(0,\omega\setminus\set{\ell},0,1)\,.
 \]
As $x\leq q_2=(0,0,1,0)$, it follows that $x\leq 0_K$, \contr.
\end{proof}

\section{Join-semidistributivity of the lattices $\rK(D)$}\label{S:JsdyK(D)}

In this section we shall prove that all lattices~$\rK(D)$ are \jsd. In fact we shall prove a stronger statement, namely that all lattices~$\rK(D)$ satisfy the lattice identity (inclusion) \SDjn{3}. The optimality of the ``$3$'' superscript is a consequence of the following easy statement.

\begin{proposition}\label{P:notSD+2}
The lattice $\rK(B)$ does not satisfy \SDjn{2}, for any Boolean lattice~$B$ with more than two elements.
\end{proposition}

\begin{proof}
There are nonzero complementary elements $a,b\in B$. We set
 \begin{align*}
 x&:=(1,a,0,0)\,,\\
 y&:=(0,b,0,1)\,,\\
 z&:=(0,a,1,0)\,. 
 \end{align*}
It is easy to compute that $\xp_2(x,y,z)=(0,b,0,0)$ while $x\vee(y\wedge z)=(1,a,0,0)$, so $\xp_2(x,y,z)\nleq x\vee(y\wedge z)$.
\end{proof}

\begin{proposition}\label{P:K(D)SD+3}
The lattice $\rK(D)$ satisfies \SDjn{3}, for any nontrivial \bdl~$D$.
\end{proposition}

\begin{proof}
  We must prove that $\xp_3(x,z,y)\leq x\vee(y\wedge z)$ for all
  $x,y,z\in\rK(D)$. By replacing~$x$ by $x\vee(y\wedge z)$, then~$z$
  by $z\wedge(x\vee y)$, we may assume without loss of generality that
  \begin{equation}\label{Eq:yzleqxzleqx+y}
    y\wedge z\leq x\quad\text{and}\quad z\leq x\vee y\,.
  \end{equation}
  Elementary computations show that if two
  of the elements~$x$, $y$, $z$ are comparable, then $\xp_2(x,y,z)\leq
  x$, thus $\xp_3(x,z,y)\leq x$. Hence we shall assume from now on that
  \begin{equation}\label{Eq:antichain}
    x\,,\ y\,,\ z\ \text{ are pairwise incomparable}\,. 
  \end{equation}
  If $\rK(D)\dnw(x\vee y)$ is distributive, then, as $z\leq x\vee y$,
  we get that $z=(z\wedge x)\vee(z\vee y)\leq x$,
  contradicting~\eqref{Eq:antichain}. If $\rK(D)\dnw(x\vee z)$ is
  distributive, then, setting $y':=\xp_2(x,y,z)=y\wedge(x\vee z)$, we
  get that $y'=(y'\wedge x)\vee(y'\wedge z)\leq x$, and thus
  $\xp_3(x,z,y)\leq x$. 
  Hence we may assume that neither
  $\rK(D)\dnw(x\vee y)$ nor $\rK(D)\dnw(x\vee z)$ is distributive. By
  Lemma~\ref{L:NonDistrId}, this implies that
 \begin{equation}\label{Eq:xyxzlarge}
 \set{x\vee y,x\vee z}\subseteq\set{(1,1,1,0),(0,1,1,1),(1,1,1,1)}\,.
 \end{equation}
 Suppose first that $x_0=1$. If $x_3=1$, then, as $x\in\rK(D)$, $x=1$,
 which contradicts~\eqref{Eq:antichain}; so $x_3=0$. As $z\leq x\vee
 y$ we obtain that $z_3\leq y_3$, but $y\wedge z\leq x$ and thus
 $z_3=z_3\wedge y_3\leq x_3$, and so $z_3=0$. If $x_2=1$ then, as
 $x_0=1$ and $x<1$, we get that $x=(1,1,1,0)\geq z$, which
 contradicts~\eqref{Eq:antichain}; hence $x_2=0$. So far we have
 proved that $x=(1,x_1,0,0)$. Now from~\eqref{Eq:xyxzlarge} it follows
 that $y\nleq(1,1,0,0)$ and $z\nleq(1,1,0,0)$. If either $y_2=z_2=1$
 or $y_3=z_3=1$, then either $y_2\wedge z_2=1$ or $y_3\wedge z_3=1$,
 \contr\ as $y\wedge z\leq x$ and $x=(1,x_1,0,0)$. As $z_3\leq y_3$,
 the only remaining possibility is $y_3=z_2=1$ and $y_2=z_3=0$. As
 $y_3=1$ and $y\in\rK(D)$ it follows that $y_0=0$. Similarly, as
 $z_2=1$, $z\in\rK(D)$, and $x\nleq z$, we obtain that $z_0=0$.
 We have thus proved that
 \begin{align*}
 x&=(1,x_1,0,0)\,,\\
 y&=(0,y_1,0,1)\,,\\
 z&=(0,z_1,1,0)\,. 
 \end{align*}
It follows that $x\vee z=(1,1,1,0)$, then that $y\wedge(x\vee z)=(0,y_1,0,0)$, then that $x\vee(y\wedge(x\vee z))=(1,x_1\vee y_1,0,0)$. Therefore, $\xp_3(x,z,y)=(0,z_1\wedge(x_1\vee y_1),0,0)$. As~$D$ is distributive, $z_1\wedge x_1\leq x_1$, and $y_1\wedge z_1\leq x_1$, we obtain that $\xp_3(x,z,y)\leq x$.

The remaining case is $x_0=0$. As $z\leq x\vee y$ we get $z_0\leq
y_0$, thus $z_0=y_0\wedge z_0\leq x_0=0$, and thus $z_0=0$, and
therefore $(x\vee z)_0=0$. By~\eqref{Eq:xyxzlarge}, it follows that
$x\vee z=(0,1,1,1)$. In particular,
\begin{equation}\label{Eq:x1z1x3z31}
  x_1\vee z_1=x_3\vee z_3=1\,.
\end{equation}
If $x\vee y=(1,1,1,0)$, then $x_3=y_3=0$, thus (as $z\leq x\vee y$)
$z_3=0$, thus, by~\eqref{Eq:x1z1x3z31}, $x_3=1$, contradicting  $x\vee y=(1,1,1,0)$. If $x\vee y=(0,1,1,1)$, then $x_1\vee y_1=x_3\vee y_3=1$, thus,
by~\eqref{Eq:x1z1x3z31} and as $y\wedge z\leq x$, we get
\[
x_1=x_1\vee(y_1\wedge z_1)=(x_1\vee y_1)\wedge(x_1\vee z_1)=1\,,
\]
and similarly, $x_3=1$, so $x\geq(0,1,1,1)\geq z$, contradicting~\eqref{Eq:antichain}.

The only remaining possibility is $x_0=0$ and $x\vee y=(1,1,1,1)$. {}From $x_0=0$ it follows that $y_0=1$. If $y_3=1$ then $y=1$, which contradicts~\eqref{Eq:antichain}; so $y_3=0$. As $x\vee y=(1,1,1,1)$, we get $x_3=1$. If $y_2=1$, then, as $y_0=1$ and $y_3=0$, we get $y=(1,1,1,0)$; as $y\wedge z\leq x$ and $x_3=1$, we get $z\leq x$, which contradicts~\eqref{Eq:antichain}; so $y_2=0$. So far we have obtained that
 \begin{align*}
 x&=(0,x_1,x_2,1)\,,\\
 y&=(1,y_1,0,0)\,,\\
 z&=(0,z_1,z_2,z_3)\,. 
 \end{align*}
It follows that $x\vee z$ has the form $(0,x_1\vee z_1,*,*)$, thus
 \[
 y\wedge(x\vee z)=(0,y_1\wedge(x_1\vee z_1),0,0)
 \]
with $y_1\wedge(x_1\vee z_1)=(y_1\wedge x_1)\vee(y_1\wedge z_1)\leq x_1$, so $\xp_2(x,y,z)=y\wedge(x\vee z)\leq x$, and so $\xp_3(x,z,y)\leq x$. 
\end{proof}

As \SDjn{3} is a lattice-theoretical identity implying \jsdy, we obtain the following corollary.

\begin{corollary}\label{C:K(D)SD+3}
The class $\setm{\rK(D)}{D\text{ bounded distributive lattice}}$ generates a variety of \jsd\ lattices.
\end{corollary}

\section{Discussion}\label{S:Discussion}

\subsection{Frink's Embedding Theorem}
This theorem (cf. Frink~\cite{Frink46}, see also Gr\"atzer \cite[Section~IV.5]{GLT2}) provides a $0,1$-lattice embedding construction of any complemented modular lattice into some geomodular (i.e., algebraic, atomistic, and modular) lattice. J\'onsson observed in~\cite{Jons54} that the lattice constructed by Frink generates the same variety as~$L$.

Let us outline Frink's construction. We start with a complemented modular lattice~$L$. Denote by~$\bL:=\Fil L$ the filter lattice of~$L$ and by~$\bP$ the set of all atoms of~$\bL$ (i.e., the elements of~$\bP$ are the maximal proper filters of~$L$). By the general properties of modular lattices (cf. Crawley and Dilworth~\cite[Section~4.1]{CrDi73}), $\bP^\vee$ is an ideal of~$\bL$; in particular, it satisfies all the identities satisfied by~$\bL$, which are the same as all the identities satisfied by~$L$. Now~$\Id\bP^\vee$ is an algebraic, atomistic, modular (that is, geomodular) lattice, belonging to the variety generated by~$L$, with set of atoms $\ol{\bP}:=\setm{\bL\dnw\fp}{\fp\in\bP}$. Then~$\ol{\bP}$, being the set of all atoms in an algebraic atomistic lattice, is a seed (cf. Proposition~\ref{P:AlgAt2StrSp}). In order to conclude the proof, Frink proves that the canonical map $\eps\colon L\hookrightarrow\Id\bP^\vee$, $x\mapsto\setm{\fp\in\bP^\vee}{x\in\fp}$ is a lattice homomorphism. The hard core of that task amounts to proving the following:

\begin{quote}\em
Let $x,y\in L\setminus\set{0}$ and let~$\fr\in\bP$ such that $x\vee y\in\fr$. Then there are $\fp,\fq\in\bP$ such that $x\in\fp$, $y\in\fq$, and $\fp\cap\fq\subseteq\fr$.
\end{quote}

Conceivably, the proof of Frink's Theorem could have been dealt with in a more expeditious manner if it had been possible to prove that~$\Fil L$ is complemented. However, as Frink observes on Frink~\cite[page~462]{Frink46}, this is not the case as a rule, for example for~$L$ a continuous geometry. Later on the same page, Frink expresses the hope that ideals and filters could be combined in the same construction. We quote the corresponding paragraph on Frink~\cite[page~462]{Frink46}:
\begin{quote}\em
To be sure, one could consider the dual concept of the lattice~$D(L)$, namely the lattice of all ideals \pup{multiplicative ideals} instead of dual ideals. Among these would be found maximal ideals, which would serve as complements of maximal dual ideals, if both kinds of ideal could be combined in a single lattice.
\end{quote}
One possible illustration of the hopelessness of the desired task is the following result (cf. Wehrung~\cite[Corollary~12.4]{DLLB}): \emph{The subspace lattice of an infinite-dimensional vector space cannot be embedded into any bi-algebraic lattice}.

Another illustration of the interplay between rings and lattices about those questions is given by Faith's example (cf. Cozzens and Faith~\cite[Example~5.14]{CoFa75}) of a von Neumann regular ring~$R$ which is not a ``left V-ring'', the latter meaning that not every left ideal is an intersection of maximal left ideals. It is then not hard to prove that if~$L$ denotes the lattice of all principal right ideals of~$R$, which is complemented modular (cf. Maeda~\cite[Section~VI.4]{Maed} or Goodearl~\cite[Theorem~2.3]{Good91}), then~$\Fil L$ is not atomistic. It follows that the set of all atoms of~$\Fil L$ is not join-dense in~$\Fil L$ (thus it is not a quasi-seed) and that not all points of~$\Fil L$ are atoms (for~$\Fil L$ is spatial).

Frink's example and Faith's example are both coordinatizable (i.e., each of them is isomorphic to the lattice of all principal right ideals of a regular ring), but they are not identical. While Frink's example is a continuous geometry, Faith starts with an infinite-dimensional right vector space~$V$ over a field~$F$, then considers the ideal~$S$ in the endomorphism ring $\End V_F$ consisting of all endomorphisms with finite rank, then defines~$R$ as the subalgebra of~$\End V_F$ generated by~$S$ and the identity.

\subsection{Word problems in various classes of lattices}
The geometric description provided for $n$-distributive lattices in Theorem~\ref{T:NatRefndistr} is the key tool for establishing the result, stated in Theorem~\ref{T:Finn-distr}, that the variety of all $n$-distributive lattices is generated by its finite members, and thus has a decidable word problem for free lattices. On the other hand, the corresponding results for \emph{modular} lattices do not hold (cf. Freese~\cite{Freese79,Freese80} and Herrmann~\cite{Herr84}). It is even observed on Freese~\cite[page~90]{Freese80} that the free lattice on five generators in the variety of all $n$-distributive modular lattices, for $n\geq4$, has an undecidable word problem. As every modular lattice embeds into some algebraic, modular, and spatial lattice (Herrmann, Pickering, and Roddy~\cite{HPR}), the result of Theorem~\ref{T:Finn-distr} for $n$-distributive lattices appears a bit as a fluke. A look at the axiomatization of the abstract projective geometries associated with algebraic spatial lattices described in Herrmann, Pickering, and Roddy~\cite[Section~3]{HPR}, in particular the so-called \emph{Triangle Axiom}, shows that the existential quantifier involved in that axiom prevents us from expressing an infinite projective space as a ``limit'' of finite projective spaces in any satisfactory way. Due to Lemma~\ref{L:Pjoinndistr}, this obstacle does not appear in the case of $n$-distributive lattices.

However, a further look at \emph{positive} decidability results obtained for other classes of modular lattices shows that this existential quantifier alone is not sufficient to prevent decidability to occur. Here is a sample of such results, the third one being of more hybrid nature due to the extra operation symbol for complementation:

\begin{itemize}
\item[---] Hutchinson and Cz\'edli \cite{HuCz78} characterize those rings~$R$ for which the word problem for free lattices in the variety generated by all subspace lattices of left $R$-modules is decidable. This class of rings includes all fields, and also the ring~$\ZZ$ of all integers as well as its quotient rings $\ZZ/m\ZZ$ for positive integers~$m$ (cf. Herrmann~\cite{Herr73}, Herrmann and Huhn~\cite{HeHu75}).

\item[---] Herrmann and Huhn~\cite{HeHu75} also prove that the word problem for free lattices in the variety generated by all complemented modular lattices is solvable.

\item[---] By using results about von~Neumann regular rings proved by Goodearl, Menal, and Moncasi~\cite{GoMM93}, Herrmann and Semenova~\cite{HeSe07} prove that the variety generated by complemented Arguesian lattices with an extra unary operation symbol for complementation is generated by its finite members, and thus that the word problem for free lattices with complementation in the variety generated by those structures is decidable. The latter decidability result extends to the variety generated by all complemented modular lattices with a unary operation symbol for complementation, although residual finiteness is replaced by residual finite length.
\end{itemize}

This suggests that in the modular world, the spatial theory alone is probably far from sufficient for settling residual finiteness and word problem matters.

\subsection{Open problems}\label{S:Pbs}

\begin{problem}\label{Pb:AlgStrSp}
Can every algebraic and spatial lattice be embedded, within its variety, into some algebraic and strongly spatial lattice?
\end{problem}

For our next problem, we shall consider the identities $\beta'_m$, given in Nation~\cite{Nation90}, that characterize, among finite lattices, those lattices without $\bD$-sequences of length $m+1$ (where~$\bD$ denotes join-dependency). Nation proves in \cite[Section~5]{Nation90} that for fixed~$m$, the variety of lattices defined by~$\beta'_m$ is locally finite.

\begin{problem}\label{Pb:betam}
Prove that every lattice satisfying~$\beta'_m$ for some~$m$ can be embedded, within its variety, into some algebraic and strongly spatial lattice.
\end{problem}

Our next problem asks for a semidistributive analogue of Theorem~\ref{T:NoVarEmb}.

\begin{problem}\label{Pb:SemidNonEmb}
Construct a semidistributive lattice that cannot be embedded, within its variety, into any algebraic spatial lattice. Can such a lattice be locally finite, or generate a variety of semidistributive lattices, or both?
\end{problem}

\begin{problem}\label{Pb:SpMod}
Let~$p$ be a compact element in an algebraic, modular, spatial lattice~$L$. Can every join-cover of~$p$ be refined to a minimal join-cover of~$p$?
\end{problem}

By Theorem~\ref{T:ModSp2StrSp}, the conclusion of Problem~\ref{Pb:SpMod} holds in case~$p$ is a point of~$L$.

\begin{problem}\label{Pb:Quasivar}
Can every modular lattice be embedded, within its \emph{quasivariety}, into some algebraic (spatial, algebraic spatial, respectively) lattice?
\end{problem}

\end{document}